\documentclass[11pt,a4paper]{article}
\pdfoutput=1 
\usepackage[]{geometry}
\usepackage[mathcal]{euscript}
\usepackage{bm,url}                     
\usepackage{amsfonts}
\usepackage{amssymb}
\usepackage{amsmath}
\usepackage{amsthm}
\usepackage[utf8x]{inputenc}
\usepackage{float}
\usepackage[pdftex]{graphicx}
\DeclareGraphicsExtensions{.bmp,.png,.pdf,.jpg,.ps}
\usepackage{natbib}                 
\usepackage{colortbl}  
\usepackage{todonotes}             
\usepackage{booktabs}

\newtheorem{theorem}{Theorem}
\newtheorem{lemma}{Lemma}
\newtheorem{proposition}{Proposition}

\usepackage[english]{babel}
\usepackage{amsmath}
\usepackage{amsfonts}
\usepackage{amssymb}
\usepackage{graphicx}
\usepackage{color}
\usepackage{array}
\usepackage{mathrsfs}
\usepackage{hyperref}

\definecolor{violet}{rgb}{0.7,0,0.6}
\definecolor{OliveGreen}{RGB}{85,107,47}

\title{A Truncated Cramér-von Mises test of normality}
\author{Juan Kalemkerian\\
Universidad de la República, Facultad de Ciencias.}

\begin{document}
\maketitle 

\begin{abstract}
\noindent A new  test of  normality based on a standardised empirical process is introduced in this article.
 The first step is to introduce a Cramér--von Mises type statistic with weights equal to the inverse of the standard normal 
 density function supported
on a symmetric interval $[-a_n,a_n]$ depending on the sample size $n.$
The sequence of end points $a_n$ tends to infinity, and is chosen so that the statistic goes to infinity at 
the speed of $\ln \ln n.$
After substracting the mean, a suitable test statistic is obtained, with the same asymptotic law as the well-known Shapiro--Wilk statistic.
The performance of the new test is described and compared with three other well-known tests of normality, namely, Shapiro--Wilk,
Anderson--Darling and that of del Barrio-Matr\'an, Cuesta Albertos, and Rodr\'{\i}guez Rodr\'{\i}guez, by means of power calculations under many 
 alternative hypotheses.
\end{abstract}

\newpage 

\section{Definitions and notation}
In this section we will introduce some definitions and notation to be used throughout this paper.
The notation $:=$ means a definition.
Given $X_{1},X_{2},...,X_{n}$ iid, we define $\overline{X}%
_{n}:=\frac{X_{1}+X_{2}+...+X_{n}}{n}$ the sample mean, $S_{n}^{2}:=%
\frac{1}{n}\sum_{i=1}^{n}\left( X_{i}-\overline{X}_{n}\right) ^{2}$ the
sample variance, and $Y_{i}:=\frac{X_{i}-\overline{X}_{n}}{%
S_{n}}$ for $i=1,2,...,n$, and we define the standardised empirical process $%
\widehat{b}_{n}(x):=\frac{1}{\sqrt{n}}\sum_{i=1}^{n}\left( 1_{\left\{
Y_{i}\leq x\right\} }-\phi (x)\right)$ for all $x \in \mathbb{R}$. Also we write $F_n (x):=\frac{1}{n}\sum_{i=1}^{n} 1_{\left\{
X_{i}\leq x\right\} }$ for the empirical distribution, and 
$b_{n}(x):=\frac{1}{\sqrt{n}}\sum_{i=1}^{n}\left( 1_{\left\{
X_{i}\leq x\right\} }-\phi (x)\right)$.\\
We write $\varphi \left( x\right) :=\frac{1}{\sqrt{2\pi}}e^{\frac{-x ^{2}}{2\sigma ^{2}} }$ and $\phi \left(
x\right) :=\int_{-\infty }^{x}\varphi \left( t\right) dt$ for the density and distribution functions of $N(0,1)$.
The notation $\xrightarrow{P}$, $\xrightarrow{a.s.}$ and $\xrightarrow{w}$ means
convergence in probability, almost surely and in distribution, respectively.
 Throughout this paper, we write $a_n$ for the sequence  $a_{n}:=\phi ^{-1}\left( 1-1/n\right) .$\\

\bigskip
\section{Introduction}
\noindent  In the present paper, we study the Cram\'{e}r--von Mises statistic to test normality
when the weight function is $\psi \left( t\right) =\frac{1}{\varphi
^{2}\left( \phi ^{-1}(t)\right) }$ and we integrate over the interval $\left(
-a_{n},a_{n}\right) $ instead of over all the real line. The main contribution of this paper \ is that we will obtain the
asymptotic distribution of this `truncated' kind of Cram\'{e}r-von-Mises
statistic and we will prove that its limit distribution is equivalent to
that of the Shapiro-Wilk test. Also, for fixed sample
sizes, we will simulate the behaviour of its power, for a wide class of
alternative hypotheses, and we will see that its performance is comparable to that of the
Shapiro-Wilk and Anderson--Darling tests, improving them in many cases.\\
Given  $X_{1},X_{2},...,X_{n}$ iid with distribution $F$ when the null
hypothesis is  $F=F_{0}$ for a certain $F_{0}$ (fixed), Cram\'{e}r (1928)
postulated the statistic  \\ $n\int_{-\infty }^{+\infty }\left(
F_{n}(x)-F_{0}(x)\right) ^{2}dx$, and more
generally von Mises (1931) took \\ $n\int_{-\infty }^{+\infty }\left(
F_{n}(x)-F_{0}(x)\right) ^{2}\rho (x)dx$ when $\rho $ is some non-negative weight
function. Later Smirnov (1936,1937) proposed  working with the statistic \\ $n\int_{-\infty
}^{+\infty }\left( F_{n}(x)-F_{0}(x)\right) ^{2}\psi \left( F_{0}(x)\right)
dF_{0}(x)$, for certain $\psi .$ Nowadays all these statistics (varying $%
\psi $) are called Cram\'{e}r--von Mises statistics. The two most popular
cases are when $\psi \left( t\right) =1$ (because of its simplicity and similarity
with the original Cram\'{e}r statistic), i.e. $%
w_{n}^{2}:=n\int_{-\infty }^{+\infty }\left( F_{n}(x)-F_{0}(x)\right)
^{2}dF_{0}(x)$, which is called the Cram\'{e}r--von Mises statistic; and when $\psi
\left( t\right) =\frac{1}{t(1-t)},$ proposed by Anderson and Darling (1954),
i.e.  $A_{n}^{2}:=n\int_{-\infty }^{+\infty }\frac{%
\left( F_{n}(x)-F_{0}(x)\right) ^{2}}{F_{0}(x)\left( 1-F_{0}(x)\right) }%
dF_{0}(x)$. In this case, the quadratic difference $\left(
F_{n}(x)-F_{0}(x)\right) ^{2}$ is normalized by dividing it by its
expected value, and this is called the Anderson--Darling statistic.

$A_{n}^{2}$ has the advantage of being generally more powerful than $w_{n}^{2}$,
for a wide class of alternative hypotheses, see, for instance, Stephens (1986).

When the null hypothesis is that the distribution lies in a certain
parametric family, i.e., $F(x)=F(x,\theta )$ for certain $\theta $, then it
is natural to extend the Cram\'{e}r--von Mises statistics used for a simple
hypothesis in the form  \[n\int_{-\infty }^{+\infty }\left( F_{n}(x)-F(x,%
\widehat{\theta })\right) ^{2}\psi \left( F(x,\widehat{\theta })\right) dF(x,%
\widehat{\theta })\] where $\widehat{\theta }$ is a suitable estimator of $%
\theta $. Then we obtain  $\widehat{A}_{n}^{2}:=n\int_{-\infty }^{+\infty }%
\frac{\left( F_{n}(x)-F(x,\widehat{\theta })\right) ^{2}}{F(x,\widehat{%
\theta })\left( 1-F(x,\widehat{\theta })\right) }dF(x,\widehat{\theta })$
and $\widehat{w}_{n}^{2}:=n\int_{-\infty }^{+\infty }\left( F_{n}(x)-F(x,%
\widehat{\theta })\right) ^{2}dF(x,\widehat{\theta }).$ Again, generally, $%
\widehat{A}_{n}^{2}$ performs better than $\widehat{w}_{n}^{2}.$

The general study of empirical processes when parameters are estimated was
conducted by Durbin (1973). He used the theory of weak convergence in $D%
\left[ 0,1\right] .$ Durbin's results give the asymptotic distribution of a
wide class of Cram\'{e}r--von Mises statistics and many others.

Normality tests are a particular case of goodness of fit tests. Many tests of normality
have been designed. Stephens (1986), for instance, provides a good summary.
The theory of tests for normality was initiated by Pearson (1895, 1930), Fisher (1930)
and Williams (1935), with the $\sqrt{\beta _{1}}$ and $\beta _{2}$ tests\
(third and fourth standardised moments). Of course, the performance  and the asymptotic distribution 
of the classical Kolmogorov--Smirnov, $\chi ^{2}$, Cram\'{e}r--von Mises, Anderson--Darling and
many other tests of normality, have been extensively studied. Further,  with the Shapiro--Wilk (1965) test for
normality, there began the development of correlation and regression tests;
see, for instance, Lockhart and Stephens (1998) for a good review of this
subject. In the normal case, the Cram\'{e}r--von Mises statistics are $%
n\int_{-\infty }^{+\infty }\left( F_{n}(t)-\phi \left( \frac{t-\overline{%
X_{n}}}{S_{n}}\right) \right) ^{2}\psi \left( \phi \left( \frac{t-\overline{%
X_{n}}}{S_{n}}\right) \right) \varphi \left( \frac{t-\overline{X_{n}}}{S_{n}}%
\right) dt$.

In the normal case, there have also been extensively studied, both theoretically and
empirically, the statistics $\widehat{w}_{n}^{2}$ and $\widehat{A}_{n}^{2}.$ 

In Csörgó and Faraway (1996) the authors find an exact and asymptotic
distribution for $w_{n}^{2}.$ In Pettitt and Stephens (1976) the authors modified the  Cram\'{e}r--von
Mises statistic to allow them to test the normality of censored samples. The asymptotic theory for this is
found in Pettitt (1976). Stephens (1974), found that in the case of
the normality tests  $\widehat{w}_{n}^{2}$ and $\widehat{A}_{n}^{2}$,
especially the latter, are comparable in terms of power with the Shapiro--Wilk
test.

 The paper is organized as follows. In
Section 3, we will examine the reasons for taking this particular weight
function $\psi $.  In Section  4 we will
find the asymptotic distribution of the proposed test and give a proof that
it is equivalent to the asymptotic distribution of the Shapiro-Wilk test. Our concluding remarks
are in section 5. In Appendix 1 we will calculate the power of this test
under certain alternative hypotheses, and we will compare it with other
tests, like the Shapiro--Wilk, Anderson-Darling, and others. In Appendix 2 we give indications for calculating the 
statistic and in Appendix 3 we give tables with their critical values for different possible levels of significance and for different values of sample size $n$.

\section{Test Aims and Approach}
Given $X_1,X_2,...,X_n$ iid random variables with distribution $F$, we want to test
$H_0: F(x)=\phi \left ( \frac{x-\mu}{\sigma} \right )$ for certain $\mu $ and $\sigma$.
In del Barrio et al. (1999), a test of goodness of fit for normal
distributions based on the Wasserstein distance is presented, and they
prove that that test statistic has an asymptotic distribution equivalent to
that of the Shapiro--Wilk test (1965). This limiting distribution
is $\int_{-\infty }^{+\infty }\frac{b^{2}(x)-\mathbb{E}b^{2}(x)}{\varphi (x)}%
dx-\left( \int_{-\infty }^{+\infty }xb(x)dx\right) ^{2}-\left( \int_{-\infty
}^{+\infty }b(x)dx\right) ^{2}$ where $\{b(x)\}_{x\in \mathbb{R}}$ is a $%
\phi -$Brownian bridge. The first summand of this decomposition allows us to propose 
$%
T_n^*:=\int_{-a_{n}}^{a_{n}}\frac{\widehat{b}_{n}^{2}(x)}{\varphi (x)}dx$ to test normality.
 Although $%
\int_{-\infty }^{+\infty }\frac{\widehat{b}_{n}^{2}(x)}{\varphi (x)}%
dx<+\infty $ for each $n$, Kalemkerian (2016) proves that 
$
\int_{-a_{n}}^{a_{n}}\frac{\mathbb{E}\left( \widehat{b}_{n}^{2}(x)\right) } {%
\varphi (x)}dx=\ln \left( \ln n\right) +c_{n}, \text{ where }\left\{
c_{n}\right\}$ is a convergent sequence. This result is similar to the equality 
$\int_{-a_{n}}^{a_{n}}\frac{\mathbb{E}b^{2}(x) }{\varphi \left( x\right) }%
dx=\ln \left( \ln n\right) +\ln 2+\gamma +o\left( 1\right) $ where $\gamma $
is Euler's constant, proved in de Wet and Venter (1972).
Then $\int_{-a_{n}}^{a_{n}}\frac{\widehat{b}_{n}^{2}(x)}{\varphi
(x)}dx$ goes  to infinity as $n\rightarrow \infty.$ Hence we cannot expect $T_n^*$ to converge.
So if we follow the ideas of del
Barrio et al. (1999), we can propose studying \[\int_{-a_{n}}^{a_{n}}\frac{\widehat{b}%
_{n}^{2}(x)-\mathbb{E}\widehat{b}_{n}^{2}(x)}{\varphi (x)}dx \text{ or } \int_{-a_{n}}^{a_{n}}\frac{\widehat{b}%
_{n}^{2}(x)-\mathbb{E}{b}^{2}(x)}{\varphi (x)}dx\] as a test statistic with the intention of obtaining
an asymptotic distribution, which means that the rejection zone is \[\left \{ \int_{-a_{n}}^{a_{n}}\frac{\widehat{b}%
_{n}^{2}(x)-\mathbb{E}\widehat{b}_{n}^{2}(x)}{\varphi (x)}dx \geq c \right \} \text{ or } \left \{ \int_{-a_{n}}^{a_{n}}\frac{\widehat{b}%
_{n}^{2}(x)-\mathbb{E}{b}^{2}(x)}{\varphi (x)}dx \geq c \right \}.\]
  Then for fixed $n$, we propose $%
\int_{-a_{n} }^{a_{n} }\frac{\widehat{b}_{n}^{2}(x)}{\varphi (x)}%
dx$ instead $\int_{-a_{n}}^{a_{n}}\frac{\widehat{b}%
_{n}^{2}(x)-\mathbb{E}\widehat{b}_{n}^{2}(x)}{\varphi (x)}dx$ or $\int_{-a_{n}}^{a_{n}}\frac{\widehat{b}%
_{n}^{2}(x)-\mathbb{E}{b}^{2}(x)}{\varphi (x)}dx$ to test normality.
Finally, if $t=S_{n}x+\overline{X_{n}}$, then $\widehat{b}_{n}(x)=%
\sqrt{n}\left( F_{n}\left( t\right) -\phi \left( \frac{t-\overline{X_{n}}}{%
S_{n}}\right) \right) $, so $T_{n}^{\ast }=\int_{-a_{n}}^{a_{n}}\frac{\widehat{b}_{n}^{2}(x)}{\varphi (x)}dx=\frac{n}{%
S_{n}}\int_{\overline{X_{n}}-a_{n}S_{n}}^{\overline{X_{n}}+a_{n}S_{n}}\frac{\left( F_{n}\left( t\right) -\phi \left( 
\frac{t-\overline{X_{n}}}{S_{n}}\right) \right) ^{2}}{\varphi \left( \frac{t-%
\overline{X_{n}}}{S_{n}}\right) }dt.$ Then, estimating $\mu $
by $\overline{X_{n}}$ and $\sigma $ by $S_{n}$, it follows that $T_{n}^{\ast }$ can be seen as a kind of Cram\'{e}r--von Mises 
statistic to test normality when the weight function is $\psi (t)=\frac{1}{\varphi ^{2}\left( \phi ^{-1}(t)\right) }$ and integrated over the interval
$\left( -a_{n},a_{n}\right) $ instead of over all the real line.

\section{Asymptotic distribution of the TCVM test statistic}

\noindent In this section, given $X_1,X_2,...,X_n$ iid, we propose $\left \{ T_n \geq c \right \}$ to test normality, 
where $T_n:= \int_{-a_{n}}^{a_{n}}\frac{\widehat{b}%
_{n}^{2}(x)-\mathbb{E}{b}^{2}(x)}{\varphi (x)}dx$, and we will obtain the asymptotic distribution of $T_n$ under $H_0$.
Since $\widehat{b}_{n}(x)$ is invariant under changes of location and scale, to test normality, we may assume
(if the null hypothesis is true) that
 $X_{1},X_{2},...,X_{n} $ are independent $N\left( 0,1\right)
$. We can express $\widehat{b}_{n}(x)$ and $b_n (x)$ (defined in Section 1) as
\begin{equation*}
\widehat{b}_{n}(x)=b_{n}(x_{n})+\sqrt{n}\left( \phi (x_{n})-\phi (x)\right)
\end{equation*}%
where $x_{n}=S_{n}x+\overline{X}_{n}.$ Then, given $x,$ $\widehat{b}_{n}(x)$
converges to $b(x)$ with $\left\{ b(x)\right\} _{x\in \mathbb{R}}$ being the 
Brownian bridge associated with $\phi $,  the asymptotic limit of the processes $%
\left\{ b_{n}(x)\right\} _{x\in \mathbb{R}}$. Moreover, Kalemkerian (2016) proved that%
\begin{equation*}
\int_{-a_{n}}^{a_{n}}\frac{\mathbb{E}\left( \widehat{b}_{n}^{2}(x)\right) }{%
\varphi (x)}dx=\ln \left( \ln n\right) +c_{n}\text{ where }\left\{
c_{n}\right\} \text{ converges.}
\end{equation*}%
Hence, our test statistic goes to infinity.
In Lemma 4 we give meaning to the expression $%
\int_{-\infty }^{+\infty }\frac{b^{2}(x)-\phi (x)\left( 1-\phi (x)\right) }{%
\varphi (x)}dx$, therefore recalling that $\mathbb{E}\left( b^{2}(x)\right) =\phi
(x)\left( 1-\phi (x)\right) $, it is reasonable to assume that if we subtract
from our statistic its expected value or an equivalent quantity, we  should find a non-trivial
asymptotic distribution of the sequence \begin{equation*}
T_{n}=\int_{-a_{n}}^{a_{n}}\frac{\widehat{b}_{n}^{2}(x)-\mathbb{E}\left(
b^{2}(x)\right) }{\varphi (x)}dx.
\end{equation*}%
Indeed, in the following theorem, we summarize the main result of this section.
\begin{theorem}

$T_n$ converges in law to 
\begin{equation*}
\int_{-\infty }^{+\infty }\frac{b^{2}(x)-\mathbb{E}\left( b^{2}(x)\right) }{%
\varphi (x)}dx-\left ( \int_{-\infty}^{+\infty}xdb(x) \right )^{2}-\left ( \int_{-\infty}^{+\infty}x^{2}db(x) \right )^{2}+c,
\end{equation*}%
where $c$ is a constant and $\{b(x)\}_{x \in \mathbb{R}} $ is a $\phi$- Brownian bridge.

\end{theorem}
To obtain this result, we use the Skorokhod construction (Skorokhod, 1956). Thus, there exists a triangular array of row 
independent random variables $X_{n1},...,X_{nn}$ and $\{ b(x)\}_{x \in \mathbb{R}}$ a $\phi$- Brownian bridge, all defined on the same sample space, such that 
$||b_n-b||\xrightarrow{a.s.}0$ where $||f||=\sup_{x \in \mathbb{R}}|f(x)|$ and $\{b_n(x)\}_{x \in \mathbb{R}}$ are the empirical 
processes asociated to $X_{n1},...,X_{nn}$.
To simplify the notation, we will continue to write $X_1,...,X_n$ instead $X_{n1},...,X_{nn}$ throughout the rest of this paper.

Theorem 1 follows from the following succession of lemmas and propositions.

We begin by noting that
\begin{equation}\label{bn2}
\widehat{b}_{n}^{2}(x)=b_{n}^{2}(x_{n})+n\left( \phi (x_{n})-\phi (x)\right)
^{2}+2\sqrt{n}b_{n}(x_{n})\left( \phi (x_{n})-\phi (x)\right),
\end{equation}
 adding and substracting the expression  $\mathbb{E}\left(b_{n}^{2}(x_{n}) \right )$ 
 to the numerator of the  integrand in $T_n$, we obtain the
following decomposition

\begin{equation}\label{Rn}
T_{n}=H_{n}+I_{n}+J_{n}
\end{equation}%
where 
\begin{equation*}
H_{n}:=\int_{-a_{n}}^{a_{n}}\frac{b_{n}^{2}(x_{n})-\mathbb{E}\left(
b_{n}^{2}(x_{n})\right) }{\varphi (x)}dx,
\end{equation*}%
\begin{equation*}
I_{n}:=\int_{-a_{n}}^{a_{n}}\frac{n\left( \phi (x_{n})-\phi (x)\right) ^{2}+2%
\sqrt{n}b_{n}(x_{n})\left( \phi (x_{n})-\phi (x)\right) }{\varphi (x)}dx,
\end{equation*}%
\begin{equation*}
J_{n}:=\int_{-a_{n}}^{a_{n}}\frac{\mathbb{E}\left( b_{n}^{2}(x_{n})\right) -%
\mathbb{E}\left( b^{2}(x)\right) }{\varphi (x)}dx.
\end{equation*}%

\begin{lemma}
 \begin{equation} \label{equivalencia de phi}
  \frac{x\varphi (x)}{1+x^{2}}\leq \phi \left( -x\right) \leq \frac{\varphi (x)%
}{x}\text{ for all }x>0.
\end{equation}
\end{lemma}

\begin{proof}
$\varphi \left( x\right) =\int_{x}^{+\infty }t\varphi \left( t\right) dt\geq
x\int_{x}^{+\infty }\varphi \left( t\right) dt=x\left( 1-\phi \left(
x\right) \right) .$ Then, the second inequality in (\ref{equivalencia de phi}) is true. For the
other inequality, observe that $1-\phi \left( x\right) =\int_{x}^{+\infty
}\varphi \left( t\right) dt=\int_{x}^{+\infty }t\varphi \left( t\right) 
\frac{1}{t}dt$. If we integrate by parts, we obtain that $1-\phi \left(
x\right) =\frac{\varphi \left( x\right) }{x}-\int_{x}^{+\infty }\frac{%
\varphi \left( t\right) }{t^{2}}dt\geq \frac{\varphi \left( x\right) }{x}-%
\frac{1-\phi \left( x\right) }{x^{2}}$, then $1-\phi \left( x\right) \geq 
\frac{x\varphi (x)}{1+x^{2}}.$
\end{proof}

\begin{lemma}
$\frac{a_{n}^{p}}{%
n^{q}}\rightarrow 0$ for all $p,q>0.$ 
\end{lemma}
\begin{proof}
From the previous lemma, we obtain that if $x\rightarrow +\infty $, then $\phi
\left( -x\right) \sim \frac{\varphi \left( x\right) }{x},$ where  $f(x)\sim
g(x)$ means that $\frac{f(x)}{g(x)}\rightarrow 1$ when $x\rightarrow +\infty
.$ Then \[
\frac{a_{n}^{p}}{n^{q}}=\phi ^{q}\left( -a_{n}\right) a_{n}^{p}\sim \varphi
^{q}\left( a_{n}\right) a_{n}^{p-1}\rightarrow 0.\]\end{proof}

\begin{lemma} For any $\{b(x)\}_{x \in \mathbb{R}}$ $\phi-$Brownian bridge, it holds that 
 \begin{equation}
  \int_{-\infty}^{+\infty}\int_{-\infty}^{+\infty}\frac{\mathbb{COV}\left(
b^{2}(x);b^{2}(y)\right) }{\varphi (x)\varphi (y)}dxdy<+\infty.
 \end{equation}

\end{lemma}
\begin{proof} We use that if $(X,Y)$ are centered Gaussian random variables in $\mathbb{R}^{2}$, then 
$\mathbb{COV}(X^{2},Y^{2})=2 \left ( \mathbb{COV}(X,Y) \right )^{2}$. Then 

 \begin{equation*}
\int_{-\infty}^{+\infty}\int_{-\infty}^{+\infty}\frac{\mathbb{COV}\left(
b^{2}(x);b^{2}(y)\right) }{\varphi (x)\varphi (y)}dxdy=2\int_{-\infty}^{+\infty}%
\int_{-\infty}^{+\infty}\frac{\phi ^{2}(x\wedge y)\left( 1-\phi \left( x\vee
y\right) \right) ^{2}}{\varphi (x)\varphi (y)}dxdy=
\end{equation*}%
\begin{equation*}
4\int_{-\infty}^{+\infty}dx\int_{-\infty}^{x}\frac{\phi ^{2}(y)\left( 1-\phi
\left( x\right) \right) ^{2}}{\varphi (x)\varphi (y)}dy=4%
\int_{-\infty}^{+\infty}\frac{\left( 1-\phi \left( x\right) \right) ^{2}}{%
\varphi (x)}dx\int_{-\infty}^{x}\frac{\phi ^{2}(y)}{\varphi (y)}dy.
\end{equation*}

If $x\rightarrow -\infty $, then 
\begin{equation*}
\frac{\left( 1-\phi \left( x\right) \right) ^{2}}{\varphi (x)}%
\int_{-\infty}^{x}\frac{\phi ^{2}(y)}{\varphi (y)}dy\leq \frac{\left( 1-\phi
\left( x\right) \right) ^{2}}{\varphi (x)}\int_{-\infty}^{x}\frac{\varphi (y)%
}{y^{2}}dy\leq 
\end{equation*}%
\begin{equation*}
\frac{1}{\varphi (x)}\frac{\phi (x)}{x^{2}}\leq \frac{1}{%
\left\vert x^{3}\right\vert }.
\end{equation*}%
If $x\rightarrow +\infty $, then 
\begin{equation*}
\frac{\left( 1-\phi \left( x\right) \right) ^{2}}{\varphi (x)}%
\int_{-\infty}^{x}\frac{\phi ^{2}(y)}{\varphi (y)}dy\leq \frac{\left( 1-\phi
\left( x\right) \right) ^{2}}{\varphi (x)}\left[ \int_{-\infty }^{1}\frac{%
\phi ^{2}(y)}{\varphi (y)}dy+\int_{1}^{x}\frac{\phi ^{2}(y)}{\varphi (y)}dy%
\right] .
\end{equation*}%
To complete the proof, it is enough to observe that 
\begin{equation*}
\frac{\left( 1-\phi \left( x\right) \right) ^{2}}{\varphi (x)}\int_{1}^{x}%
\frac{\phi ^{2}(y)}{\varphi (y)}dy\leq \frac{\varphi (x)}{x^{2}}\int_{1}^{x}%
\frac{dy}{\varphi (y)}\leq 
\end{equation*}%
\begin{equation*}
\frac{\varphi (x)}{x^{2}}\int_{1}^{x}%
\frac{ydy}{\varphi (y)}=\frac{1-e^{-(x^{2}-1)/2}}{x^{2}} \leq \frac{1}{x^{2}}.
\end{equation*}

\end{proof}

\begin{lemma}
Define the sequence \begin{equation}
                        W_{n}:=\int_{-a_{n}}^{a_{n}}\frac{b^{2}(x)-\mathbb{E}\left(
b^{2}(x)\right) }{\varphi (x)}dx=\int_{-a_{n}}^{a_{n}}\frac{b^{2}(x)-\phi
(x)\left( 1-\phi (x)\right) }{\varphi (x)}dx.
                       \end{equation}
 Then exists a random variable $%
Y$ in $L^{2}$ such that $\mathbb{E}\left( W_{n}-Y\right) ^{2}\rightarrow 0.$
\end{lemma}

\begin{proof}
We will demonstrate that $\{W_n\}$ is an $L^{2}$-Cauchy sequence.

\begin{equation*}
\mathbb{E}\left( W_{n}^{{}}W_{m}\right)
=\int_{-a_{n}}^{a_{n}}\int_{-a_{m}}^{a_{m}}\mathbb{E}\left( \frac{%
b^{2}(x)-\phi (x)\left( 1-\phi (x)\right) }{\varphi (x)}\frac{b^{2}(y)-\phi
(y)\left( 1-\phi (y)\right) }{\varphi (y)}\right) dxdy=
\end{equation*}%
\[
\int_{-a_{n}}^{a_{n}}\int_{-a_{m}}^{a_{m}}\frac{\mathbb{E}\left(
b^{2}(y)b^{2}(x)\right) -\phi \left( x\right) \phi \left( y\right)
\left( 1-\phi \left( x\right) \right) \left( 1-\phi \left( y\right) \right) 
}{\varphi \left( x\right) \varphi \left( y\right) }dxdy
\]
which goes to  \[
\int_{-\infty}^{+\infty}\int_{-\infty}^{+\infty}\frac{\mathbb{COV}\left(
b^{2}(y)b^{2}(x)\right)}{\varphi \left( x\right) \varphi \left( y\right) }dxdy < +\infty
\]for $n,m \rightarrow +\infty$.\\
Then $\mathbb{E}(W_n-W_m)^{2}=\mathbb{E}(W_n^{2}+W_m^{2}-2W_nW_m) \rightarrow 0,$ and $\left\{ W_{n}\right\} $ is an $%
L^{2}-$Cauchy sequence. \end{proof}

From now, we call $Y:=\int_{-\infty}^{+\infty}\frac{b^{2}(x)-\phi(x)(1-\phi(x))}{\varphi(x)}dx.$

\begin{lemma} 
There exist $Z_1$ and $Z_2$ centered, independent Gaussian random variables such that
\begin{equation*}\label{Z1 y Z2}
\left (\sqrt{n}\overline{X}_{n} , \sqrt{n}\left (S_{n}-1\right ) \right )
\xrightarrow{P} \left ( Z_1 , Z_2 \right ).
\end{equation*}
\end{lemma}

\begin{proof}
We \ define $Y_{n}:=\sup_{x\in \mathbb{R}}\frac{\left\vert
b_{n}(x)-b(x)\right\vert }{\sqrt{\varphi \left( x\right) }}$, and we will apply
theorem 1.1 of Shorack \& Wellner (1982) to $q(t)=\sqrt{\varphi \left( \phi
^{-1}\left( t\right) \right) }$. Set $\| f \|= \sup_{x \in [0,1]} |f(x)|$, and $B_n(t):=b_n \left ( \phi^{-1}(t) \right)$, 
$B(t):=b \left( \phi^{-1}(t) \right)$. Then  $Y_n= \left \| \frac{B_n-B}{q} \right \|$. Using Lemma 2, that $\frac{%
q^{2}(t)}{t\ln \left( -\ln t\right) }\rightarrow +\infty $ for $%
t\rightarrow 0,$ then $Y_{n}\overset{P}{\rightarrow }0.$\\
$\int_{-\infty}^{+\infty}|x|^{\alpha}|b_n(x)-b(x)|dx \leq Y_n \int_{-\infty}^{+\infty}|x|^{\alpha}\sqrt{\varphi(x)}dx 
\xrightarrow{P} 
0 $ for all $\alpha>0.$ Then, 

\noindent $\sqrt{n}\overline{X}_{n}=\sqrt{n}\int_{-\infty }^{+\infty }xd\left(
F_{n}(x)-\phi (x)\right) =\int_{-\infty }^{+\infty }xdb_{n}(x)\xrightarrow{P}
\int_{-\infty }^{+\infty }xdb(x):=Z_{1}.$

\noindent  $\sqrt{n}(S_n-1)=\frac{\sqrt{n}(S_{n}^{2}-1)}{S_n+1}=\frac{1}{S_n+1} \left (
\int_{-\infty }^{+\infty }x^{2}db_{n}(x)-\sqrt{n}\overline{X}_{n}^{2} \right ) \xrightarrow{P} \frac{1}{2}\int_{-\infty
}^{+\infty }x^{2}db(x):=Z_{2}. $\\
Observe that $Z_1$, $Z_2$ are centered Gaussian random variables, and $\mathbb{V}(Z_1)=1, $ $\mathbb{V}(Z_2)=1/2.$\\
As $\sqrt{n}\overline{X}_{n}$ and $S_n$ are independent for all $n$, it follows that $Z_1$ and $Z_2$ are independent too.\end{proof}

\begin{lemma}
 \[
 \mathbb{E}\left( b_{n}^{2}(x)b_{n}^{2}(y)\right) = \mathbb{E}\left( b^{2}\left( x\right) b^{2}(y)\right) +\frac{p(x,y)}{n}
\] where \[p(x,y)= \allowbreak \phi \left( x\wedge y\right) \left(
1-\phi \left( x\vee y\right)\right) \frac{1+6\phi \left( x\vee y\right)\phi \left( x\wedge y\right)-2\phi \left( x\vee y\right)-4\phi \left( x\wedge y\right)  }{n}\] 
 and $\left\{ b\left( x\right) \right\} _{x\in \mathbb{R}}$ is any 
$\phi -$Brownian motion.
\end{lemma}

\begin{proof}

For simplicity in notation, we will call $z=\phi \left( x\wedge y\right) $
and $t=\phi \left( x\vee y\right) .$
\[
 \mathbb{E}\left( b_{n}^{2}(x)b_{n}^{2}(y)\right) =n^{2}\mathbb{E}\left(
F_{n}(x)-\phi (x)\right) ^{2}\left( F_{n}(y)-\phi (y)\right) ^{2}=
\]

\begin{equation}\label{bcuadradocuadrado}
 n^{2}\mathbb{E}\left( F_{n}^{2}(x)-2F_{n}(x)\phi (x)+\phi ^{2}\left(
x\right) \right) \left( F_{n}^{2}(y)-2F_{n}(y)\phi (y)+\phi ^{2}\left(
y\right) \right) 
\end{equation}

To calculate (\ref{bcuadradocuadrado}) we will use that $\mathbb{E}\left( F_{n}(x)\right) =\phi \left( x\right) $
and $\mathbb{E}\left( F_{n}^{2}(x)\right) =\frac{\phi \left( x\right) \left(
1-\phi \left( x\right) \right) }{n}+\phi ^{2}(x)$ for all $x,$ and \ the
formulas for $\mathbb{E}\left( F_{n}(x)F_{n}(y)\right) ,$ $\mathbb{E}\left(
F_{n}^{2}(x)F_{n}(y)\right) $ and $\mathbb{E}\left(
F_{n}^{2}(x)F_{n}^{2}(y)\right) $ that we will obtain. 
First, we observe that \begin{equation} \label{Fcuadradocuadrado}
                        n^{2}\mathbb{E}\left( F_{n}(x)F_{n}(y)\right) =\sum_{i,j=1}^{n}P\left(
X_{i}\leq x,X_{j}\leq y\right) =nz+n(n-1)zt.
                       \end{equation}
                       
If we define $%
I:=-2\phi \left( y\right) \mathbb{E}\left( F_{n}^{2}(x)\phi \left( y\right)
\right) -2\phi \left( x\right) \mathbb{E}\left( F_{n}^{2}(y)\phi \left(
x\right) \right) $ then (\ref{bcuadradocuadrado}) (after some calculation) is
equal to 
\[
n^{2}\mathbb{E}F_{n}^{2}(x)F_{n}^{2}(y)+n^{2}I+5n
z^{2}t+nzt^{2} +\left( 3n^{2}-6n\right) z^{2}t^{2}.
\]
If we procced similarly to (\ref{Fcuadradocuadrado}), we obtain that
\[
n^{2}\mathbb{E}\left( F_{n}^{2}(x)F_{n}(y)\right) =\frac{1}{n}%
\sum_{i,j,k=1}^{n}P\left( X_{i}\leq x,X_{j}\leq x,X_{k}\leq y\right) =
\]%
\[
\left( n-1\right) \left( n-2\right) \phi ^{2}(x)\phi \left( y\right)
+(n-1)\phi \left( x\right) \phi \left( y\right) +2(n-1)\phi \left( x\wedge
y\right) \phi \left( x\right) +\phi \left( x\wedge y\right) ,
\]%
and%
\[
n^{2}\mathbb{E}\left( F_{n}^{2}(x)F_{n}^{2}(y)\right) =\frac{1}{n^{2}}%
\sum_{i,j,k,l=1}^{n}P\left( X_{i}\leq x,X_{j}\leq x,X_{k}\leq y,X_{l}\leq
y\right) 
\]%
\[
\frac{n-1}{n}\left(
4z^{2}+3zt+5(n-2)z^{2}t+(n-2)zt^{2}+(n^{2}-5n+6)z^{2}t^{2}\right) +z/n
\]%
Then 
\[
n^{2}I=-4\left( n^{2}-3n+2\right) z^{2}t^{2}+\left( 10-10n\right)
z^{2}t+\left( 2-2n\right) zt^{2}-2z^{2}-2zt.
\]%
Finally, $\mathbb{E}\left( b_{n}^{2}(x)b_{n}^{2}(y)\right) $ is equal to 
\[
-5tz^{2}-t^{2}z+3t^{2}z^{2}+tz+2z^{2}+z\left( t-1\right) \frac{2t+4z-6tz-1}{n%
}=\]
\begin{equation}\label{bnb}
 \mathbb{E}\left( b^{2}\left( x\right) b^{2}(y)\right) +\allowbreak z\left(
t-1\right) \frac{2t+4z-6tz-1}{n}.
\end{equation}
where $\left\{ b\left( x\right) \right\} _{x\in \mathbb{R}}$ is any 
$\phi -$Brownian motion.
\end{proof}

\begin{lemma}
 \begin{equation}
  \int_{-a_{n}}^{a_{n}}\frac{%
b_{n}^{2}(x)-\phi (x)\left( 1-\phi \left( x\right) \right) }{\varphi \left(
x\right) }dx\overset{P}{\rightarrow }\int_{-\infty }^{+\infty }\frac{%
b^{2}(x)-\phi (x)\left( 1-\phi \left( x\right) \right) }{\varphi \left(
x\right) }dx.
 \end{equation}

\end{lemma}

\begin{proof}
 We decompose 
\[
\int_{-a_{n}}^{a_{n}}\frac{b_{n}^{2}(x)-\phi (x)\left( 1-\phi
(x)\right) }{\varphi (x)}dx:=Y_{n,\delta }+Z_{n,\delta }+W_{n,\delta } 
\]%
where 
$
Y_{n,\delta }:=\int_{-\delta }^{\delta }\frac{b_{n}^{2}(x)-\phi (x)\left(
1-\phi (x)\right) }{\varphi (x)}dx\text{, }Z_{n,\delta
}:=\int_{-a_{n}}^{-\delta }\frac{b_{n}^{2}(x)-\phi (x)\left( 1-\phi
(x)\right) }{\varphi (x)}dx\text{ and } \\ W_{n,\delta }:=\int_{\delta }^{a_{n}}%
\frac{b_{n}^{2}(x)-\phi (x)\left( 1-\phi (x)\right) }{\varphi (x)}dx. 
$
We will prove that $\lim_{\delta \rightarrow +\infty }\lim_{n\rightarrow
+\infty }Y_{n,\delta }=Y$, \\ $\lim_{\delta \rightarrow +\infty
}\lim_{n\rightarrow +\infty }Z_{n,\delta }=0$ and $\lim_{\delta \rightarrow
+\infty }\lim_{n\rightarrow +\infty }W_{n,\delta }=0$ in probability.

Indeed 
\[
\lim_{\delta \rightarrow +\infty }\lim_{n\rightarrow +\infty }Y_{n,\delta
}=\lim_{\delta \rightarrow +\infty }\int_{-\delta }^{\delta }\frac{%
b^{2}(x)-\phi (x)\left( 1-\phi (x)\right) }{\varphi (x)}dx= \] \[\int_{-\infty
}^{+\infty }\frac{b^{2}(x)-\phi (x)\left( 1-\phi (x)\right) }{\varphi (x)}%
dx. 
\]%
To prove that $\lim_{\delta \rightarrow +\infty }\lim_{n\rightarrow +\infty
}W_{n,\delta }=0$ in probability, it is enough to show that $\lim_{\delta
\rightarrow +\infty }\lim_{n\rightarrow +\infty }\mathbb{E}\left(
W_{n,\delta }^{2}\right) =0.$

 \[ \mathbb{E}\left( W_{n,\delta }^{2}\right) =\mathbb{E}\left( \int_{\delta
}^{a_{n}}\frac{b_{n}^{2}(x)-\phi (x)\left( 1-\phi (x)\right) }{\varphi (x)}%
dx\int_{\delta }^{a_{n}}\frac{b_{n}^{2}(y)-\phi (y)\left( 1-\phi (y)\right) 
}{\varphi (y)}dy\right) = \] 
\begin{equation}\label{note}
\int_{\delta }^{a_{n}}\int_{\delta }^{a_{n}}\frac{%
\mathbb{E}\left( b_{n}^{2}(x)b_{n}^{2}(y)\right) -\phi (x)\left( 1-\phi
(x)\right) \phi (y)\left( 1-\phi (y)\right) }{\varphi (x)\varphi (y)}dxdy
\end{equation}
Then, using Lemma 6 we obtain that (\ref{note}) is equal to

\[
\int_{\delta }^{a_{n}}\int_{\delta }^{a_{n}}\frac{\mathbb{COV}\left(
b^{2}(x),b^{2}(y)\right) }{\varphi (x)\varphi (y)}dxdy+\frac{1}{n}%
\int_{\delta }^{a_{n}}\int_{\delta }^{a_{n}}\frac{p\left( x,y\right) }{%
\varphi (x)\varphi (y)}dxdy. 
\]

From Lemma 3 we obtain that 
\[
\lim_{\delta \rightarrow +\infty }\lim_{n\rightarrow +\infty }\int_{\delta
}^{a_{n}}\int_{\delta }^{a_{n}}\frac{\mathbb{COV}\left(
b^{2}(x),b^{2}(y)\right) }{\varphi (x)\varphi (y)}dxdy=0. 
\]

Finally, we will show that $\lim_{n\rightarrow +\infty }\frac{1}{n}%
\int_{\delta }^{a_{n}}\int_{\delta }^{a_{n}}\frac{p\left( x,y\right) }{%
\varphi (x)\varphi (y)}dxdy=0$ for all $\delta >0.$

Indeed, for certain $c>0$, 
\[
\frac{1}{n}\int_{\delta }^{a_{n}}\int_{\delta }^{a_{n}}\frac{\left\vert
p\left( x,y\right) \right\vert }{\varphi (x)\varphi (y)}dxdy\leq \frac{c}{n}%
\int_{\delta }^{a_{n}}\int_{\delta }^{a_{n}}\frac{\phi \left( x\wedge
y\right) \left( 1-\phi \left( x\vee y\right) \right) }{\varphi (x)\varphi (y)%
}dxdy. 
\]%
Then, applying L'H\^{o}pital rule, we obtain that 
\[
\lim_{n\rightarrow +\infty }\frac{1}{n}\int_{\delta }^{a_{n}}\int_{\delta
}^{a_{n}}\frac{\phi \left( x\wedge y\right) \left( 1-\phi \left( x\vee
y\right) \right) }{\varphi (x)\varphi (y)}dxdy= \] \[\lim_{n\rightarrow +\infty }%
\frac{\partial }{\partial n}\left( \int_{\delta }^{a_{n}}\int_{\delta
}^{a_{n}}\frac{\phi \left( x\wedge y\right) \left( 1-\phi \left( x\vee
y\right) \right) }{\varphi (x)\varphi (y)}dxdy\right) = 
\]

\[
\lim_{n\rightarrow +\infty }\frac{2}{n^{2}\varphi ^{2}\left( a_{n}\right) }%
\int_{\delta }^{a_{n}}\frac{\phi \left( x\right) \left( 1-\phi \left(
a_{n}\right) \right) }{\varphi \left( x\right) }dx=\lim_{n\rightarrow
+\infty }\frac{2}{n^{3}\varphi ^{2}\left( a_{n}\right) }\int_{\delta
}^{a_{n}}\frac{\phi \left( x\right) }{\varphi \left( x\right) }dx=0. 
\]

The case $\lim_{\delta \rightarrow +\infty }\lim_{n\rightarrow +\infty
}Z_{n,\delta }=0$ is solved similarly.

\end{proof}

In the next lemmas and propositions, we will use the following equalities:
\begin{equation}\label{Taylor}
 \phi \left( x_{n}\right) -\phi \left( x\right) =\varphi \left( c_{n}\right)
\left( x_{n}-x\right) \text{ for }c_{n}\in \left( x_{n}\wedge x,x_{n}\vee
x\right).
\end{equation}

\begin{equation}\label{unomas}
 \frac{\varphi ^{2}\left( c_{n}\right) }{\varphi ^{2}\left( x\right) }%
=e^{x^{2}-c_{n}^{2}}=1+e^{d_{n}}\left( x^{2}-c_{n}^{2}\right) \text{ for }%
\left\vert d_{n}\right\vert \leq \left\vert x^{2}-c_{n}^{2}\right\vert .
\end{equation}

\begin{equation}\label{unomas/2}
 \frac{\varphi \left( c_{n}\right) }{\varphi \left( x\right) }=e^{\frac{%
x^{2}-c_{n}^{2}}{2}}=1+e^{d_{n}^{\prime }}\frac{\left(
x^{2}-c_{n}^{2}\right) }{2}\text{ for }\left\vert d_{n}^{\prime }\right\vert
\leq \frac{\left\vert x^{2}-c_{n}^{2}\right\vert }{2}.
\end{equation}

Observe that if $x\in \left( -a_{n},a_{n}\right) $, then $\left\vert
d_{n}\right\vert \leq \left\vert x^{2}-x_{n}^{2}\right\vert \leq \frac{%
Aa_{n}^{2}}{\sqrt{n}}$ for certain random variable $A,$ and $\left\vert
d'_{n}\right \vert \leq \frac{\left\vert x^{2}-x_{n}^{2} \right \vert}{2} \leq \frac{%
A'a_{n}^{2}}{\sqrt{n}}$ for certain random variable $A'.$

In the next proposition, we prove that the term $H_n$ in decomposition (\ref
{Rn}) converges in law.

\begin{proposition}
\begin{equation*}
H_{n}=\int_{-a_{n}}^{a_{n}}\frac{b_{n}^{2}(x_{n})-\mathbb{E}\left(
b_{n}^{2}(x_{n})\right) }{\varphi (x)}dx \xrightarrow{P} \int_{-\infty
}^{+\infty }\frac{b^{2}(x)-\mathbb{E}\left( b^{2}(x)\right) }{\varphi (x)}dx.
\end{equation*}

\end{proposition}

\begin{proof}
Since 
\begin{equation*}
\int_{-a_{n}}^{a_{n}}\frac{b_{n}^{2}(x_{n})-\mathbb{E}\left(
b_{n}^{2}(x_{n})\right) }{\varphi (x_{n})}dx=\frac{1}{S_{n}}%
\int_{-S_{n}a_{n}+\overline{X}_{n}}^{S_{n}a_{n}+\overline{X}_{n}}\frac{%
b_{n}^{2}(t)-\mathbb{E}\left( b_{n}^{2}(t)\right) }{\varphi (t)}dt
\end{equation*}%
converges in probability (by Lemma 7) to $\int_{-\infty }^{+\infty }\frac{b^{2}(x)-\mathbb{E}\left(
b^{2}(x)\right) }{\varphi (x)}dx$,
\begin{equation*}
H_{n}=\int_{-a_{n}}^{a_{n}}\frac{b_{n}^{2}(x_{n})-\mathbb{E}\left(
b_{n}^{2}(x_{n})\right) }{\varphi (x_{n})}\frac{\varphi \left( x_{n}\right) 
}{\varphi (x)}dx
\end{equation*}%
and due to (\ref{unomas/2}) (with $c_{n}=x_{n}$) and Lemma 3 we obtain that $H_{n}$ converges in probability to $%
\int_{-\infty }^{+\infty }\frac{b^{2}(x)-\mathbb{E}\left( b^{2}(x)\right) }{%
\varphi (x)}dx$.
\end{proof}

\begin{lemma}
\[
\int_{-a_{n}}^{a_{n}}\frac{n\left( \phi \left( x_{n}\right) -\phi \left(
x\right) \right) ^{2}}{\varphi \left( x\right) }dx=n\left( S_{n}-1\right)
^{2}\alpha _{n}+n\overline{X}_{n}^{2}\beta _{n}+Y_{n}
\]%
where $\alpha _{n}=\int_{-a_{n}}^{a_{n}}x^{2}\varphi (x)dx,$ $\beta
_{n}=\int_{-a_{n}}^{a_{n}}\varphi (x)dx$ and $\mathbb{E}\left( \left\vert
Y_{n}\right\vert \right) \rightarrow 0.$
\end{lemma}

\begin{proof}
Using (\ref{Taylor}) and (\ref{unomas}), we obtain that 
\[
\int_{-a_{n}}^{a_{n}}\frac{n\left( \phi \left( x_{n}\right) -\phi \left(
x\right) \right) ^{2}}{\varphi \left( x\right) }dx=\int_{-a_{n}}^{a_{n}}%
\frac{n\varphi ^{2}\left( c_{n}\right) \left( x_{n}-x\right) ^{2}}{\varphi
^{2}\left( x\right) }\varphi (x)dx=
\]%
\[
\int_{-a_{n}}^{a_{n}}n\left( x_{n}-x\right) ^{2}\varphi
(x)dx+\int_{-a_{n}}^{a_{n}}ne^{d_{n}}\left( x^{2}-c_{n}^{2}\right) \left(
x_{n}-x\right) ^{2}\varphi (x)dx.
\]%
\[
\int_{-a_{n}}^{a_{n}}n\left( x_{n}-x\right) ^{2}\varphi
(x)dx=\int_{-a_{n}}^{a_{n}}n\left( \left( S_{n}-1\right) x+\overline{X}%
_{n}\right) ^{2}\varphi (x)dx= \] \[ n\left( S_{n}-1\right) ^{2}\alpha _{n}+n%
\overline{X}_{n}^{2}\beta _{n}.
\]
Define $Y_{n}:=\int_{-a_{n}}^{a_{n}}ne^{d_{n}}\left( x^{2}-c_{n}^{2}\right)
\left( x_{n}-x\right) ^{2}\varphi (x)dx.$ Then, using that $|d_n|\leq \frac{Aa_{n}^{2}}{\sqrt{n}}$ and Lemma 3, 
\[
\mathbb{E}\left( \left\vert Y_{n}\right\vert \right) \leq \mathbb{E}\left(
\int_{-a_{n}}^{a_{n}}ne^{d_{n}}\left\vert x^{2}-c_{n}^{2}\right\vert \left(
x_{n}-x\right) ^{2}\varphi (x)dx\right) \leq \frac{a_{n}^{2}}{\sqrt{n}}%
\mathbb{E}\left( Ae^{\frac{Aa_{n}^{2}}{\sqrt{n}}}\right) \rightarrow 0.
\]%
\end{proof}
\newpage
\begin{lemma}
\begin{equation*}
\int_{-a_{n}}^{a_{n}}\frac{n\left( \phi (x_{n})-\phi (x)\right) ^{2}}{%
\varphi (x)}dx\overset{P}{\rightarrow }Z_{1}^{2}+Z_{2}^{2}.
\end{equation*}
\end{lemma}

\begin{proof}
This follows directly from Lemma 8 and Lemma 5.
\end{proof}

\begin{lemma}

\[
\int_{-a_{n}}^{a_{n}}\frac{\sqrt{n}b_{n}\left( x_{n}\right) \left( \phi
\left( x_{n}\right) -\phi \left( x\right) \right) }{\varphi \left( x\right) }%
dx=\sqrt{n}\frac{\left( S_{n}-1\right) }{S_{n}^{2}}Z_{n}^{\prime }+\sqrt{n}%
\frac{\overline{X}_{n}}{S_{n}^{2}}Z_{n}^{\prime \prime }+Y_{n}^{\prime }
\]%
where $Z_{n}^{\prime }=\int_{\overline{X}_{n}-a_{n}S_{n}}^{\overline{X}%
_{n}+a_{n}S_{n}}tb_{n}(t)dt,$ $Z_{n}^{\prime \prime }=\int_{\overline{X}%
_{n}-a_{n}S_{n}}^{\overline{X}_{n}+a_{n}S_{n}}b_{n}(t)dt$ and $\mathbb{E}%
\left( \left\vert Y_{n}^{\prime }\right\vert \right) \rightarrow 0.$
\end{lemma}

\begin{proof}
Using  (\ref{Taylor}) and (\ref{unomas/2}), observe that  
\[
\int_{-a_{n}}^{a_{n}}\frac{\sqrt{n}b_{n}\left( x_{n}\right) \left( \phi
\left( x_{n}\right) -\phi \left( x\right) \right) }{\varphi \left( x\right) }%
dx=\int_{-a_{n}}^{a_{n}}\frac{\sqrt{n}b_{n}\left( x_{n}\right) \varphi
\left( c_{n}\right) \left( x_{n}-x\right) }{\varphi \left( x\right) }dx=
\]%
\[
\int_{-a_{n}}^{a_{n}}\sqrt{n}b_{n}\left( x_{n}\right) \left( x_{n}-x\right)
dx+\frac{1}{2}\int_{-a_{n}}^{a_{n}}\sqrt{n}b_{n}\left( x_{n}\right)
e^{d_{n}^{\prime }}\left( x^{2}-c_{n}^{2}\right) \left( x_{n}-x\right) dx.
\]%
\begin{equation*}
\sqrt{n}\int_{-a_{n}}^{a_{n}}b_{n}(x_{n})(x_{n}-x)dx=\frac{\sqrt{n}}{%
S_{n}^{2}}\int_{-S_{n}a_{n}+\overline{X}_{n}}^{S_{n}a_{n}+\overline{X}%
_{n}}b_{n}(t)\left[ \left( S_{n}-1\right) t+\overline{X}_{n}\right] dt=
\end{equation*}%
\begin{equation*}
\frac{\sqrt{n}\left( S_{n}-1\right) }{S_{n}^{2}}\int_{-S_{n}a_{n}+\overline{X%
}_{n}}^{S_{n}a_{n}+\overline{X}_{n}}tb_{n}(t)dt+\frac{\sqrt{n}\overline{X}%
_{n}}{S_{n}^{2}}\int_{-S_{n}a_{n}+\overline{X}_{n}}^{S_{n}a_{n}+\overline{X}%
_{n}}b_{n}(t)dt.
\end{equation*}
Define $Y_{n}^{\prime }:=\frac{1}{2}\int_{-a_{n}}^{a_{n}}\sqrt{n}b_{n}\left(
x_{n}\right) e^{d_{n}^{\prime }}\left( x^{2}-c_{n}^{2}\right) \left(
x_{n}-x\right) dx.$ Then working similarly to Lemma 8, we obtain that 
$
\mathbb{E}\left( \left\vert Y_{n}^{\prime }\right\vert \right)  \rightarrow 0.$
\end{proof}

\begin{lemma}
\begin{equation*}
\int_{-a_{n}}^{a_{n}}\frac{\sqrt{n}b_{n}(x_{n})\left( \phi (x_{n})-\phi
(x)\right) }{\varphi (x)}dx\overset{P}{\rightarrow }-Z_{1}^{2}-Z_{2}^{2}.
\end{equation*}
\end{lemma}

\begin{proof}

\[ Z_{n}^{\prime }=\int_{\overline{X}_{n}-a_{n}S_{n}}^{\overline{X}%
_{n}+a_{n}S_{n}}tb_{n}(t)dt=
\int_{-\infty}^{+\infty}tb_{n}(t)dt+\varepsilon_{n}\] where $\varepsilon_{n}\xrightarrow{P}0.$

Integrating by parts and using Lemma 5, we obtain that 
\[ Z_{n}^{\prime }=-\frac{1}{2}\int_{-\infty}^{+\infty}t^{2}db_{n}(t)+\varepsilon_{n}
\xrightarrow{P}-Z_{2}.\]
Analogously $Z_{n}^{\prime \prime}\xrightarrow{P}-Z_{1}, $ and the result of the lemma it follows directly from Lemma 10.
\end{proof}

\newpage

\begin{proposition}
\begin{equation*}
I_{n}=\int_{-a_{n}}^{a_{n}}\frac{n\left( \phi (x_{n})-\phi (x)\right) ^{2}+2%
\sqrt{n}b_{n}(x_{n})\left( \phi (x_{n})-\phi (x)\right) }{\varphi (x)}dx%
\overset{P}{\rightarrow }-Z_{1}^{2}-Z_{2}^{2}.
\end{equation*}
\end{proposition}

\begin{proof}
\noindent This follows directly from lemmas 9 and 11.
\end{proof}

Finally, to prove that the term $J_n$ in decomposition (\ref{Rn}) converges,
we need the following lemma.

\begin{lemma} The sequence 
\begin{equation*}
\int_{-a_{n}}^{a_{n}}\frac{\mathbb{E}\left( \widehat{b}%
_{n}^{2}(x)\right) -\mathbb{E}\left( b^{2}(x)\right) }{\varphi (x)}dx\end{equation*}
converges.
\end{lemma}

\begin{proof}
By symmetry it is enough to take the integral between 1 and $a_{n}.$
By using the results in section 3 of Kalemkerian (2016), one gets
\begin{equation*}
\int_{1}^{a_{n}}\frac{\mathbb{E}\left( \widehat{b}_{n}^{2}(x)\right) -%
\mathbb{E}\left( b^{2}(x)\right) }{\varphi (x)}dx=\frac{1}{2}\ln \ln
n+c_{n}-\int_{1}^{a_{n}}\frac{\phi (x)\left( 1-\phi (x)\right) }{\varphi (x)}%
dx=
\end{equation*}%
\begin{equation*}
\int_{1}^{a_{n}}\left( \frac{1}{x}-\frac{\phi (x)\left( 1-\phi (x)\right) }{%
\varphi (x)}\right) dx+c'_{n}.
\end{equation*} 
where $\{c_{n}\}$  and $\{c'_{n}\}$  are  convergent sequences.

Define $f(x):=\frac{1-\phi (x)}{\varphi (x)}-\frac{1}{x},$ and apply
L'Hôpital rule twice to show 
\begin{equation*}
\frac{f(x)}{\frac{1}{x^{2}}}\underset{x\rightarrow +\infty }{\rightarrow }0.
\end{equation*}
\end{proof}

\begin{proposition} The sequence 
\begin{equation*}
J_{n}=\int_{-a_{n}}^{a_{n}}\frac{\mathbb{E}\left(
b_{n}^{2}(x_{n})\right) -\mathbb{E}\left( b^{2}(x)\right) }{\varphi (x)}dx%
\end{equation*}
converges.
\end{proposition}

\begin{proof}
We will see that 
\begin{equation*}
\int_{-a_{n}}^{a_{n}}\frac{\mathbb{E}\left( b_{n}^{2}(x_{n})\right) -\mathbb{%
E}\left( \widehat{b}_{n}^{2}(x)\right) }{\varphi (x)}dx\text{ and }%
\int_{-a_{n}}^{a_{n}}\frac{\mathbb{E}\left( \widehat{b}_{n}^{2}(x)\right) -%
\mathbb{E}\left( b^{2}(x)\right) }{\varphi (x)}dx\text{ are convergent.}
\end{equation*}%
The second term converges due to Lemma 12. 
Observe that from (\ref{Rn}) one gets 
\begin{equation}
 \int_{-a_{n}}^{a_{n}}\frac{\mathbb{E}\left( b_{n}^{2}(x_{n})\right) -\mathbb{%
E}\left( \widehat{b}_{n}^{2}(x)\right) }{\varphi (x)}dx=-\mathbb{E}\left ( I_{n}\right )\rightarrow 3/2
\end{equation}
due to lemmas 8 and 10.

\end{proof}

Note that the constant appearing in Theorem 1 is unknown. As seen in Lemma 12 and Proposition 3, one gets 
\[c=\lim J_n= \lim \int_{-a_n}^{a_{n}}\frac{\mathbb{E}\left( \widehat{b}_{n}^{2}(x)\right) -%
\mathbb{E}\left( b^{2}(x)\right) }{\varphi (x)}dx +3/2.\]
In Table 2, we present the results, from $m=1000$ simulations, the estimated values  of the constant $c$  for different values of $n$. 

Table 2. Estimated values of constant $c$ for different values of $n$.

$\begin{array}{|c|c|c|c|c|c|c|}
\hline
n &500 & 1000 &10000 & 20000 & 50000 & 100000   \\ 
\hline
c &0.0258 & 0.009137 & -0.022793 & -0.01752 & 0.002268 & 0.00093  \\
 \hline
 \end{array}
$

\section{Conclusions}

This paper has presented a test of normality with $\mu$ and $\sigma$ unknown, based on the statistic $%
\int_{-a_{n}}^{a_{n}}\frac{\widehat{b}_{n}^{2}(x)}{\varphi (x)}dx$, which is
a type of Cram\'er--von Mises statistic, except that it
is integrated over a particular interval, which is a function of the sample size.
This statistic was inspired by del Barrio et al.
(1999). We analysed its asymptotic behavior, which has a limit distribution
equivalent to that of S--W statistic. A comparative study of its behaviour over a
wide range of alternative hypotheses found that the test proposed here
is often better than the tests of Shapiro--Wilk and Anderson--Darling, which are
two tests that have very good performance as tests of normality. 

\section{Appendix 1. Performance of the proposed test}

\noindent We present a table, for sample size of $n=50$,
where we compare the performance of our test, namely the Truncated Cramér--von Mises test (TCVM),
with the following four tests: 1) the non-truncated Cramér--von Mises test, which integrates over all the real line (CVM);
2) the test proposed by del Barrio, Matr\'{a}n-Cuesta, and Rodr\'{\i}guez (BCMR); 3) the Anderson--Darling test
(AD) and 4) the Shapiro--Wilk test (SW).\\
The power of these tests was compared based on 10,000 replications. We chose the level of significance $\alpha=0.05$
and the empirical critical values were obtained after 50,000 replications.\\
The list of 35 alternative hypotheses considered in the present section are borrowed from Gan and Koelher (1990) and are a summary of 
a total of 69.\\
We use the notation of Gan and Koelher (1990), and the alternative hypotheses are considered in the same order as in the cited
reference.\\
LoConN$\left( p,a\right) $ indicates a mixture of N$\left( 0;1\right) $ with
probability $1-p$ and N$\left( a;1\right) $ with probability $p.$ ScConN$%
\left( p,a\right) $ indicates the mixture of N$\left( 0;1\right) $ with
probability $1-p$ and N$\left( 0;a\right) $ with probability $p.$ TruncN$%
\left( a;b\right) $ is N$\left( 0;1\right) $ truncated to the interval $%
\left( a;b\right) .$ SB$\left( a;b\right) $ is the Johnson bounded
distribution with parameters $a,b,$ and SU$\left( a;b\right) $ is the
Johnson unbounded distribution with parameters $a,b.$ Triangle I$\left(
a\right) $ is the distribution with density $f(x)=1/a-\left\vert x\right\vert
/a^{2}$ for $\left\vert x\right\vert \leq a$, and Triangle II$\left(
a\right) $ is the distribution with density 

$\bigskip $Table 1. Powers under 35 alternative hypotheses  at the 5\% level, $n=50.$

$\fbox{$%
\begin{array}{c|c|cccccc}
\text{Set} & \text{N%
${{}^o}$%
} & \text{Alternative} & \text{TCVM} & \text{CVM} & \text{BCMR} & \text{AD}
& \text{SW} \\ 
\hline
1 & 1 & \text{LoConN}\left( 0.5;4\right)  & 0.935 & 0.432 & 0.883 & 0.956 & 0.783
\\ 
& 2 & \text{LoConN}\left( 0.5;3\right)  & 0.439 & 0.044 & 0.341 & 0.480 & 0.212 \\ 
& 3 & \text{LoConN}\left( 0.5;2\right)  & 0.084 & 0.009 & 0.053 & 0.093 & 0.033 \\ 
\hline
2 & 4 & \text{SB}\left( 0;0.5\right)  & 0.958 & 0.496 & 0.957 & 0.926 & 0.880 \\ 

& 5 & \text{Unif}\left( 0;1\right)  & 0.708 & 0.124 & 0.689 & 0.616 & 0.466
\\ 
& 6 & \text{SB}\left( 0,0.707\right)  & 0.553 & 0.063 & 0.508 & 0.495 & 0.309
\\ 
 
& 7 & \text{TruncN}\left( -1;1\right)  & 0.876 & 0.375 & 0.735 & 0.350 & 0.197
\\ 
& 8 & \text{Beta}\left( 2;2\right)  & 0.163 & 0.005 & 0.117 & 0.155 & 0.051 \\ 
& 9 & \text{Triangle I}\left( 1\right)  & 0.061 & 0.002 & 0.034 & 0.055 & 0.015 \\ 
\hline
3 & 10 & \text{t}\left( 10\right)  & 0.098 & 0.199 & 0.169 & 0.113 & 0.186 \\ 
& 11 & \text{Logist}\left( 0;1\right)  & 0.123 & 0.248 & 0.204 & 0.156 & 0.243
\\

\hline
4 & 12 & \text{ScConN}\left( 0.05;3\right)  & 0.063 & 0.114 & 0.099 & 0.071 & 
0.116 \\ 
& 13 & \text{ScConN}\left( 0.05;5\right)  & 0.114 & 0.243 & 0.207 & 0.142 & 
0.234 \\ 

\hline
5 & 14 & \text{ScConN}\left( 0.1;5\right)  & 0.168 & 0.351 & 0.298 & 0.202 & 0.340
\\ 
& 15 & \text{ScConN}\left( 0.1;7\right)  & 0.297 & 0.518 & 0.467 & 0.337 & 
0.494 \\ 
\hline
6 & 16 & \text{ScConN}\left( 0.2;3\right)  & 0.099 & 0.228 & 0.189 & 0.126 & 
0.216 \\ 

& 17 & \text{ScConN}\left( 0.2;7\right)  & 0.426 & 0.649 & 0.596 & 0.306 & 
0.464 \\ 
\hline
7 & 18 & \text{Laplace}\left( 0;1\right)  & 0.455 & 0.562 & 0.539 & 0.526 & 0.581
\\

& 19 & \text{SU}\left( 0;1\right)  & 0.688 & 0.788 & 0.770 & 0.752 & 0.808\\

& 20 & \text{t}\left( 2\right)  & 0.818 & 0.881 & 0.871 & 0.854 & 0.892 \\ 
\hline
8 & 21 & \text{Beta}\left( 2;1\right)  & 0.815 & 0.310 & 0.811 & 0.750 & 0.702 \\ 
& 22 & \text{TruncN}\left( -2;1\right)  & 0.642 & 0.242 & 0.549 & 0.620 & 0.470
\\ 
& 23 & \text{Beta}\left( 3;2\right)  & 0.227 & 0.022 & 0.177 & 0.007 & 0.095 \\ 
\hline
9 & 24 & \text{SB}\left( 1;2\right)  & 0.105 & 0.052 & 0.094 & 0.089 & 0.072 \\ 

& 25 & \text{Weibull}\left( 2\right)  & 0.344 & 0.237 & 0.394 & 0.393 & 0.355
\\ 
& 26 & \text{HalfN}\left( 0;1\right)  & 0.891 & 0.665 & 0.922 & 0.815 & 0.883
\\ 
\hline
10 & 27 & \text{LoConN}\left( 0.2;3\right)  & 0.668 & 0.369 & 0.606 & 0.589 & 
0.554 \\ 
& 28 & \text{LoConN}\left( 0.2;5\right)  & 0.999 & 0.994 & 0.999 & 0.865 & 
0.989 \\ 

\hline
11 & 29 & \text{LoConN}\left( 0.1;3\right) & 0.489 & 0.487 & 0.559 & 0.610 & 0.569 \\ 
& 30 & \text{LoConN}\left( 0.1;5\right)  & 0.960 & 0.988 & 0.986 & 0.988 & 
0.989 \\ 
\hline
12 & 31 & \text{LoConN}\left( 0.05;3\right)  & 0.263 & 0.427 & 0.402 & 0.429
& 0.418 \\ 
& 32 & \text{LoConN}\left( 0.05;5\right)  & 0.758 & 0.894 & 0.882 & 0.878 & 
0.888 \\ 
\hline
13 & 33 & \text{Triangle II}\left( 1\right)  & 0.818 & 0.314 & 0.809 & 0.499 & 
0.689 \\ 

& 34 & \chi _{4}^{2} & 0.916 & 0.818 & 0.945 & 0.921 & 0.927 \\ 

& 35 & \text{Lognormal}\left( 0;1\right)  & 0.999 & 0.999 & 0.999 & 1.000 & 1.000
\\ 
\end{array}%
$}$

\noindent $f(x)=2/a-2x/a^{2}$ for $0\leq
x\leq a.$ \\
For others values of $n$, the comparision was similar to the one presented here.
The 35 alternative hypotheses are ordered in 13 types of distributions.
Type 1 includes symmetric bimodal distributions with low
skewness. Type 2 includes symmetric multimodal distributions with
low tails. Type 3 includes distributions with slightly heavier
tails than the normal distribution. Types 4, 5 and 6 are contaminations of the
normal distribution. Type 7 includes symmetrical distributions with high
skewness. The remaining types are asymmetrical distributions. Types 8
and 9 are distributions with low skewness with opposite signs. Types
10, 11 and 12 are bimodal distributions with positive  coefficients
of symmetry. Type 13 contains distributions with extreme values of
coefficients of skewness or symmetry.
\noindent In the following graphs, we compare the powers (at the 5$\%$ level) of our test (TCVM), Anderson--Darling (AD) and 
Shapiro--Wilk (SW), for  sample sizes $n=20$ and $n=50$, when the alternative hypotheses are in the lambda Tukey family. 
The critical values were obtained by simulations, with 50,000 replications. The various powers, presented in the graphs, were found by 
simulations, with 10,000 replications.

\begin{figure}[H]
  
 \includegraphics[scale=0.4]{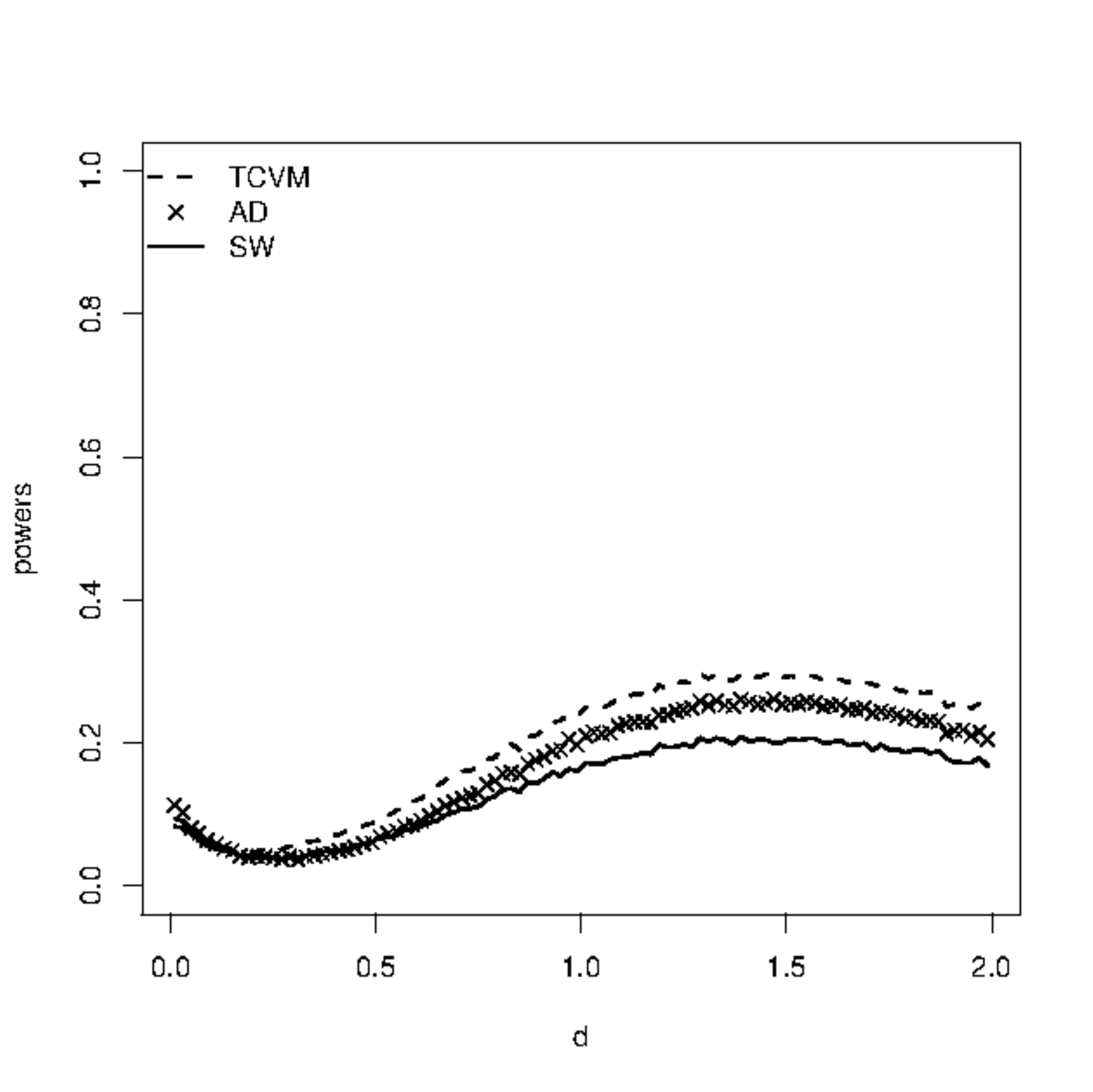} 
 \includegraphics[scale=0.355]{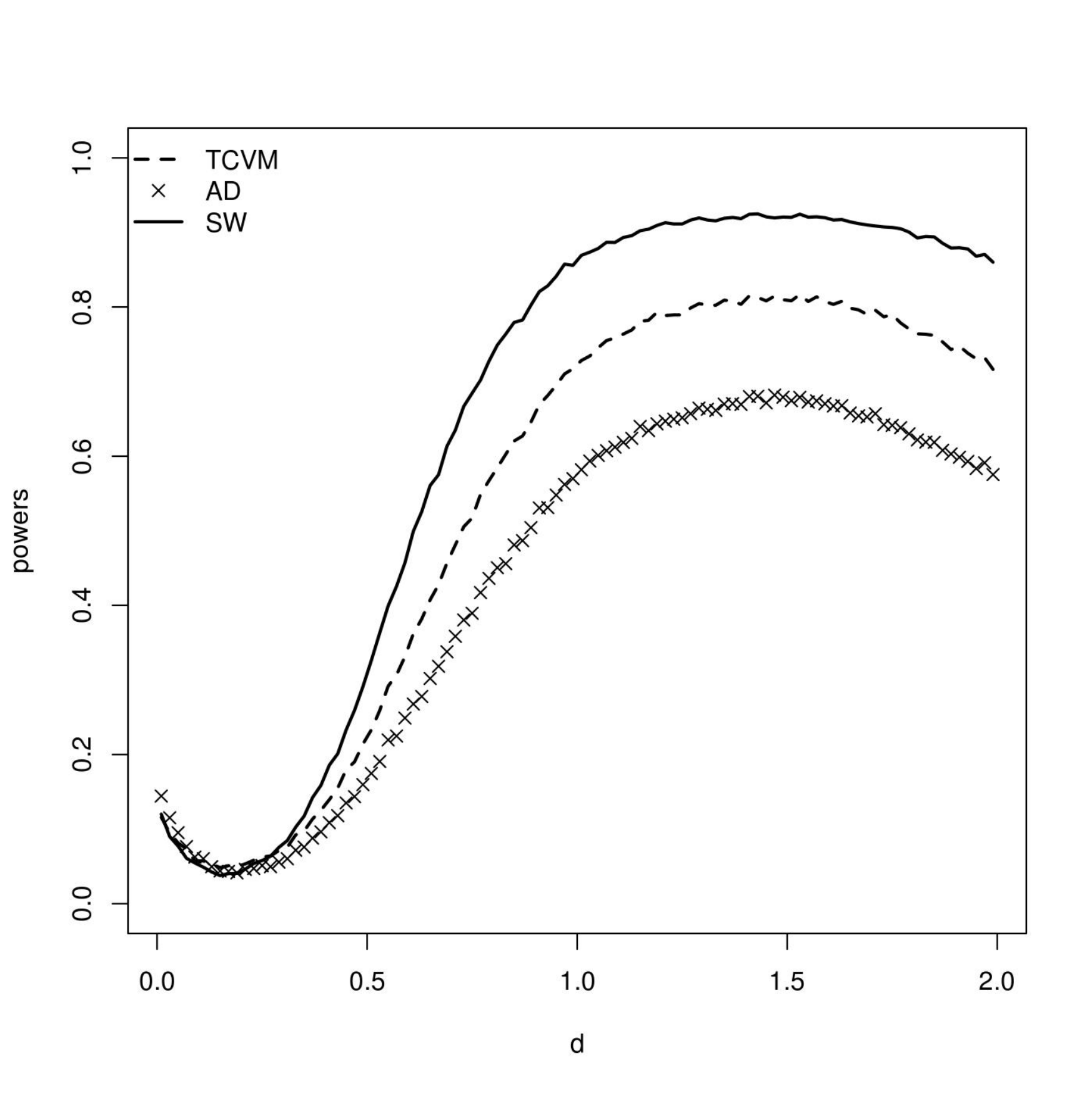}

 \caption{Tukey$(d)$ alternative, $n=20$ and $n=50$.} 
  \label{Tukey}
 
\end{figure}

\noindent In Figure ~\ref{Tukey}, the alternative hypotheses belong to Tukey's family, for sample sizes $n=20$ and $n=50$.
\noindent In Figure ~\ref{locon}, the alternative hypotheses belong to the LoConN$(0.3,p)$ family, for $n=20$ and $n=50$.
We see that for distribution Type 1, the TCVM powers are close to the AD that are the best. For 
distribution Type 2, TCVM is the test with the best performance, and CVM performs badly.
In Types 3, 4, 5 and 6, the performance is the opposite (CVM is the best test while TCVM performs badly).
We can see that in many cases, TCVM is better than the others. Besides, in several of the cases in which TCVM performs 
badly, CVM is competitive with SW, AD and BCMR. Taking into account all  35 alternative hypotheses, we can see that in many cases,
TCVM is better than the others, but under some alternative hypotheses, like symmetrical distributions with high skewness (Type 7), and 
contaminations in the variance of normal distributions (Types 4, 5 and 6), TCVM performs badly. However, in Figure ~\ref{locon}
(contaminations in the mean of normal distributions), TCVM performs very well. It should be noted that in most cases, 
either TCVM or CVM results in the best test.
\begin{figure}[H]
 \centering
   \includegraphics[scale=0.4]{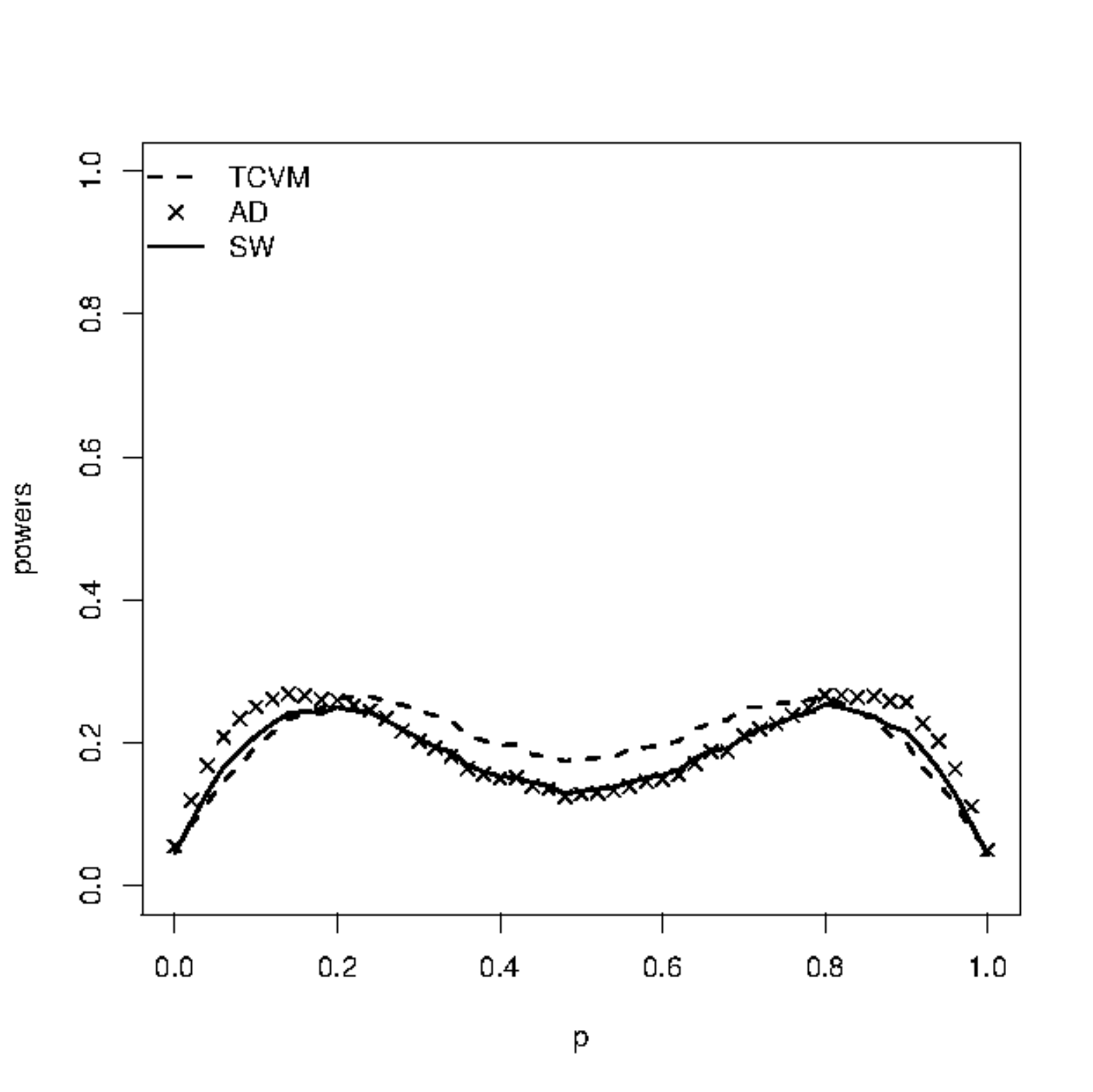} 
 \includegraphics[scale=0.4]{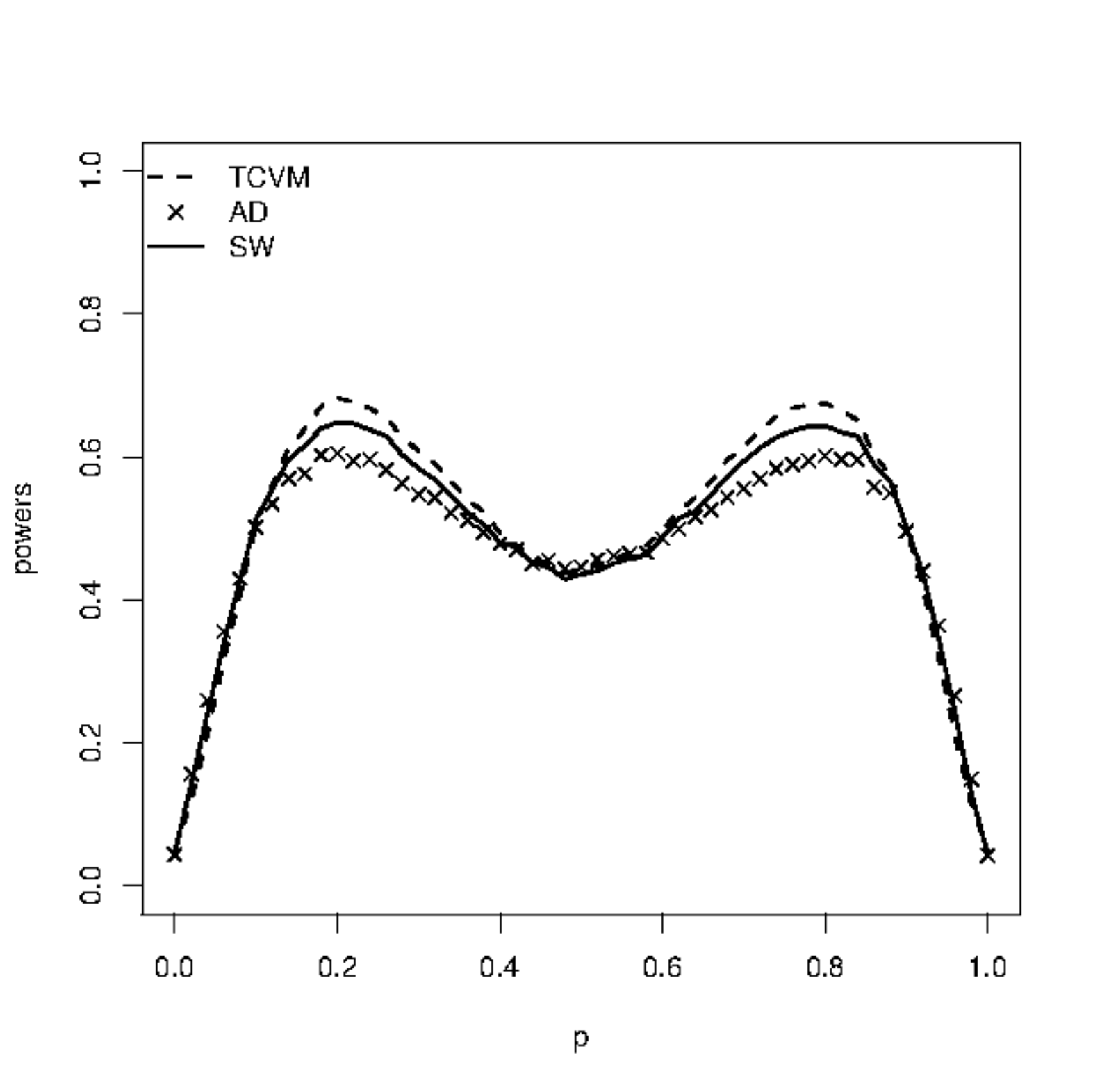} 
 \caption{LoConN$(0.3,p)$ alternative, $n=20$ and $n=50$.} 
  \label{locon}
\end{figure}

\section{Appendix 2. Computing the statistic $T_{n}^{\ast }.$}

\[
T_{n}^{\ast }=\int_{-a_{n}}^{a_{n}}\frac{\widehat{b}_{n}^{2}(x)}{\varphi (x)}%
dx=\int_{-a_{n}}^{a_{n}}\frac{\frac{1}{n}\left( \sum_{i=1}^{n}\mathbf{1}%
_{\left\{ Y_{i}\leq x\right\} }\right) ^{2}-2\phi (x)\sum_{i=1}^{n}\mathbf{1}%
_{\left\{ Y_{i}\leq x\right\} }+n\phi ^{2}(x)}{\varphi (x)}dx= 
\]%
\[
\frac{1}{n}\int_{-a_{n}}^{a_{n}}\frac{\left( \sum_{i=1}^{n}\mathbf{1}%
_{\left\{ Y_{i}\leq x\right\} }\right) ^{2}}{\varphi (x)}dx-2%
\int_{-a_{n}}^{a_{n}}\frac{\phi (x)\sum_{i=1}^{n}\mathbf{1}_{\left\{
Y_{i}\leq x\right\} }}{\varphi (x)}dx+n\int_{-a_{n}}^{a_{n}}\frac{\phi
^{2}(x)}{\varphi (x)}dx. 
\]

Then, given $X_{1},X_{2},...,X_{n}$, the calculation of the statistc $%
T_{n}^{\ast }$ can be summarized in the following steps.

\textbf{Step 1.} Compute $\overline{X}_{n},$ $S_{n}$ and delete all data such that $%
X_{i}\leq \overline{X}_{n}-a_{n}S_{n}$ or $X_{i}\geq \overline{X}%
_{n}+a_{n}S_{n}$ where the values of $a_{n}$ are shown in ninth column of Appendix 2.

\textbf{Step 2.} Define $k:=\sharp \left\{ j:\text{ }X_{j}\leq \overline{X}_{n}-a_{n}S_n\right\} $ and call
$\widetilde{X}_{1},\widetilde{X}_{2},...,\widetilde{X}_{m}$ to
the values that resulted from step 1, and sort them: $\widetilde{X}%
_{1:m}\leq \widetilde{X}_{2:m}\leq ...\leq \widetilde{X}_{m:m}$.

\textbf{Step 3.} Compute $\widetilde{Y}_{0}=-a_{n},$ $\widetilde{Y}_{m+1}=a_{n},$ $%
\widetilde{Y}_{j}=\frac{\widetilde{X}_{j:m}-\overline{X}_{n}}{S_{n}}$for $%
j=1,2,...,m.$

\textbf{Step 4.} Compute $A_{j}=\int_{\widetilde{Y}_{j}}^{\widetilde{Y}_{j+1}}\frac{dx%
}{\varphi (x)}$ and $B_{j}=\int_{\widetilde{Y}_{j}}^{\widetilde{Y}_{j+1}}%
\frac{\phi (x)dx}{\varphi (x)}$ for $j=0,1,2,...,m$ (for example in $R$ you
can use the ``integrate" function to compute $A_{j}$ and $B_{j}$ values).

\textbf{Step 5.} Compute $T_{n}^{\ast }=\frac{1}{n}\sum_{j=0}^{m}\left( j+k\right)
^{2}A_{j}-2\sum_{j=0}^{m}\left( j+k\right) B_{j}+ C_{n},$ where $C_n := n\int_{-a_{n}}^{a_{n}}%
\frac{\phi ^{2}(x)}{\varphi (x)}dx$ and their values for different sample sizes  are shown in the last column of Appendix 2.

\textbf{Step 6.} The critical values for rejection the null hypothesis for different levels of significance are in
Appendix 3. When  $T_n^{*}$ is greater than the critical value, we must reject the hypothesis of normality.

\section{Appendix 3. Tables}
In this section we present a table with critical values of $T_n^*$ for different values
of $n$. The values were obtained by simulation from 50000 replications. From columns 2 to 8, we find the levels of significance of the test. In columns 9 and 10, we find the values of $a_n$ and $C_n$ respectively.

\begin{footnotesize}
	 $%
	 \begin{array}{|c|ccccccc||c|c|} 
 \hline
\textbf{\textit{n}}  & \textbf{0.15}  & \textbf{0.1} & \textbf{0.075}  & \textbf{0.05} & 
\textbf{0.025} & \textbf{0.01}  & \textbf{0.001} & \bf{a_n} & \bf{C_n } \\ \hline
  
10  & 0.7547 & 0.8525 & 0.92590 & 1.0203 & 1.1917 & 
1.4128 & 1.9025 & 1.2816 & 28.5798 \\ 
11  & 0.7909 & 0.8961 & 0.9714 & 1.0754 & 1.2470 & 
1.4733 & 2.0513 &1.3352 &34.3184  \\ 
12  & 0.8179 & 0.9254 & 0.9996 & 1.1050 & 1.2860 & 
1.5043 & 2.0086 &1.3830 &40.4855 \\ 
13  & 0.8489 & 0.9597 & 1.0410 & 1.1456 & 1.3263 & 
1.5679 & 2.1420 & 1.4261 &47.0700  \\ 
14  & 0.8728 & 0.9840 & 1.0632 & 1.1723 & 1.3571 & 
1.5931 & 2.1457 & 1.4652 &54.0622 \\ 
15 & 0.8879 & 1.00573 & 1.08548 & 1.2007 & 1.3975 & 
1.6664 & 2.2531 &1.5341 &61.4534\\ 
16 & 0.9132 & 1.0305 & 1.1154 & 1.2308 & 1.4238 & 
1.6776 & 2.2766 &1.5011 &69.2355\\ 
17  & 0.9300 & 1.0491 & 1.1289 & 1.2381 & 1.4288 & 
1.6714 & 2.2844 &1.5647 &77.4015\\ 
18  & 0.9522 & 1.0769 & 1.1624 & 1.2814 & 1.4854 & 
1.7259& 2.4215 & 1.5932 &85.9446\\ 
19  & 0.9691 & 1.0945 & 1.1823 & 1.3062 & 1.5030 & 
1.7649 & 2.4001 &1.6199 &94.8590\\ 
20  & 0.9857 & 1.1091 & 1.1963 & 1.312 & 1.5213 & 
1.7918 & 2.4935 &1.6448 &104.1389\\ 
21  & 1.0076 & 1.1320 & 1.2174 & 1.3381 & 1.5482 & 
1.8287 & 2.4845 &1.6684 &113.7793\\ 
22  & 1.0142 & 1.1414 & 1.2330 & 1.3607 & 1.5687 & 
1.8671 & 2.5638 & 1.6906 &123.7752\\ 
23  & 1.0287 & 1.1549 & 1.2454 & 1.3729 & 1.5899 & 
1.8607 & 2.5877 &1.7117 &134.1223\\ 
24  & 1.0434 & 1.1731 & 1.2664 & 1.3940 & 1.6186 & 
1.9087 & 2.5657 &1.7317 &  144.8161\\ 
25  & 1.0538 & 1.1876 & 1.2837 & 1.4213 & 1.6322 & 
1.9017 & 2.5965 &1.7507 &155.8529\\ 
26  & 1.0587 & 1.1905 & 1.2856 & 1.4124 & 1.6442 & 
1.9479 & 2.6597 &1.7688 &167.2289\\ 
27  & 1.0745 & 1.2076 & 1.3000 & 1.4348 & 1.6762 & 
1.9522 & 2.7323 &1.7862 &178.9405\\ 
28  & 1.0847 & 1.2190 & 1.3199 & 1.4525 & 1.6863 & 
1.9890 & 2.6675 &1.8027 &190.9843\\ 
29  & 1.0942 & 1.2286 & 1.3213 & 1.4607 & 1.6994 & 
2.0166 & 2.7327 &1.8186 &203.3574\\ 
30  & 1.1102 & 1.2441 & 1.3391 & 1.4708 & 1.6967 & 
1.9752 & 2.6952 &1.8339 &216.0565\\ 
31  & 1.1192 & 1.2600 & 1.3598 & 1.5001 & 1.7279 & 
2.0439 & 2.7540 &1.8486 &229.0789\\ 
32  & 1.1226 & 1.2671 & 1.3668 & 1.5062 & 1.7311 & 
2.0420 & 2.7798 & 1.8627 &242.4218\\ 
33  & 1.1363 & 1.2760 & 1.3730 & 1.5112 & 1.7519 & 
2.0624 & 2.9159 &1.8764 &256.0826\\

 34  & 1.1408 & 1.2788 & 1.3781 & 1.5166 & 1.7584 & 
2.0704 & 2.8223 &1.8895 &270.0589 \\ 
35  & 1.1518 & 1.2886 & 1.3852 & 1.5342 & 1.7691 & 
2.1021 & 2.8338 &1.9022 &284.3482\\ 
36  & 1.1603 & 1.3033 & 1.4066 & 1.5445 & 1.7852 & 
2.1074 & 2.9445 &1.9145 &298.9482 \\ 
37  & 1.1666 & 1.3092 & 1.4067 & 1.5474 & 1.7937 & 
2.1235 & 2.8274 &1.9264 &313.8568 \\ 
38  & 1.1730 & 1.3191 & 1.4193 & 1.5632 & 1.8161 & 
2.1436 & 2.9951 &1.9379 &329.0718\\ 
39  & 1.1908 & 1.3345 & 1.4365 & 1.5804 & 1.8254 & 
2.1241 & 2.9245 &1.9491 &344.5912\\ 
40  & 1.1970 & 1.3367 & 1.4392 & 1.5841 & 1.8324 & 
2.1714 & 2.9444 &1.9600 &360.4130\\ 
41  & 1.2026 & 1.3437 & 1.4469 & 1.5920 & 1.8416 & 
2.1833 & 3.0389 & 1.9705 &376.5354\\ 
42  & 1.2111 & 1.3603 & 1.4613 & 1.6042 & 1.8393 & 
2.1652 & 2.9808 &1.9808  &392.9565\\ 
43  & 1.2015 & 1.3467 & 1.4528 & 1.6031 & 1.8512 & 
2.1814 & 2.9263 &1.9907  &409.6745\\ 
44  & 1.2213 & 1.3708 & 1.4767 & 1.6147 & 1.8585 & 
2.1950 & 3.0911 &2.0004 &426.6878\\ 
45 & 1.2285 & 1.3771 & 1.4843 & 1.6341 & 1.8750 & 
2.1900 & 3.0056 &2.0099 &443.9947\\ 
46  & 1.2350 & 1.3875 & 1.4955 & 1.6477 & 1.9053 & 
2.2467 & 3.0391 & 2.0191 &461.5935 \\ 
47  & 1.2342 & 1.3844 & 1.4903 & 1.6384 & 1.8839 & 
2.2215 & 3.0267 &2.0281 &479.4828\\ 
48  & 1.2469 & 1.4040 & 1.5141 & 1.6680 & 1.9346 & 
2.2634 & 3.0625 &2.0368 &497.6611\\ 
49 & 1.2535 & 1.4063 & 1.5159 & 1.6730 & 1.9293 & 
2.2552 & 3.0640 &2.0454 &516.1269\\ 
50  & 1.2629 & 1.4271 & 1.5304 & 1.6897 & 1.9490 & 
2.2770 & 3.1309 &2.0537 &534.8787\\ 
51  & 1.2645 & 1.4220 & 1.5301 & 1.6947 & 1.9627 & 
2.3063 & 3.2487 &2.0619 &553.9153\\ 
52  & 1.2708 & 1.4243 & 1.5399 & 1.6938 & 1.9521 & 
2.3066 & 3.1432 &2.0699 &573.2352\\ 
53  & 1.2801 & 1.4340 & 1.5445 & 1.7008 & 1.9603 & 
2.3068 & 3.1140 &2.0777 &592.8372\\ \hline
\end{array}%
$\end{footnotesize}

\begin{footnotesize}
$%
 \begin{array}{|c|ccccccc||c|c|} 
 \hline
\textbf{\textit{n}}  & \textbf{0.15}  & \textbf{0.1} & \textbf{0.075}  & \textbf{0.05} & 
\textbf{0.025} & \textbf{0.01}  & \textbf{0.001} & \bf{a_n} & \bf{C_n } \\ \hline
54  & 1.2772 & 1.4313 & 1.5422 & 1.6962 & 1.9627 & 
2.3047 & 3.1660 &2.0854 &612.7200\\ 
55  & 1.2854 & 1.4378 & 1.5531 & 1.7133 & 1.9860 & 
2.3389 & 3.2503 &2.0928 &632.8823\\ 
56  & 1.2902 & 1.4408 & 1.5486 & 1.7107 & 1.9917 & 
2.3399 & 3.2057 &2.1002 &653.3230\\ 
57  & 1.2936 & 1.4500 & 1.5653 & 1.7184 & 1.9894 & 
2.3707 & 3.2309 &2.1073 &674.0410\\ 

 58  & 1.2986 & 1.4590 & 1.5700 & 1.7210 & 1.9798 & 
2.3287 & 3.2403 &2.1144 &695.0349\\ 
59  & 1.3175 & 1.4762 & 1.5872 & 1.7379 & 2.0025 & 
2.3765 & 3.3164 &2.1213 &716.3038\\ 
60  & 1.3056 & 1.4609 & 1.5736 & 1.7346 & 2.0055 & 
2.3570 & 3.2960 &2.1280 &737.8466\\ 
61  & 1.3125 & 1.4747 & 1.5869 & 1.7424 & 2.0206 & 
2.3778 & 3.3615 &2.1347 &759.6622\\ 
62  & 1.3120 & 1.4715 & 1.5852 & 1.7434 & 2.0084 & 
2.3791 & 3.3989 &2.1412 &781.7495\\ 
63  & 1.3282 & 1.4904 & 1.5979 & 1.7568 & 2.0256 & 
2.3538 & 3.2608 &2.1476 &804.1076\\ 
64  & 1.3275 & 1.4875 & 1.6038 & 1.7623 & 2.0344 & 
2.4059 & 3.3059 &2.1539 &826.7354\\ 
65 & 1.3259 & 1.4876 & 1.6038 & 1.7701 & 2.0566 & 
2.4492 & 3.4750 &2.1600 &849.6320\\ 
66  & 1.3280 & 1.4886 & 1.6081 & 1.7691 & 2.0498 & 
2.4100 & 3.2162 &2.1661 &872.7964\\ 
67  & 1.3444 & 1.5083 & 1.6218 & 1.7782 & 2.0520 & 
2.4257 & 3.4441 &2.1721 &896.2277\\ 
68  & 1.3521 & 1.5184 & 1.6297 & 1.7866 & 2.0490 & 
2.4221 & 3.4079 &2.1779 &919.9251\\ 
69  & 1.3492 & 1.5077 & 1.6262 & 1.7902 & 2.0887 & 
2.4654 & 3.4971 &2.1837 &943.8875\\ 
70  & 1.3507 & 1.5173 & 1.6381 & 1.8052 & 2.0929 & 
2.4834 & 3.3568 &2.1893 &968.1142\\ 
71  & 1.3485 & 1.5160 & 1.6333 & 1.7975 & 2.0825 & 
2.4569 & 3.4083 &2.1949 &992.6042\\ 
72  & 1.3643 & 1.5318 & 1.6550 & 1.8260 & 2.1092 & 
2.4882 & 3.3969 &2.2004 &1017.3568\\ 

73  & 1.3668 & 1.5324 & 1.6465 & 1.8064 & 2.0829 & 
2.4538 & 3.3884 &2.2058 &1042.3711\\ 
74  & 1.3643 & 1.5258 & 1.6434 & 1.8053 & 2.0907 & 
2.4632 & 3.4550 &2.2111 &1067.6463\\ 
75  & 1.3698 & 1.5353 & 1.6487 & 1.8207 & 2.1057 & 
2.5039 & 3.3958 &2.2164 &1093.1817\\ 
76  & 1.3619 & 1.5231 & 1.6324 & 1.8007 & 2.0775 & 
2.4609 & 3.4476 &2.2215 &1118.9764\\ 
77  & 1.3931 & 1.5646 & 1.6840 & 1.8451 & 2.1306 & 
2.5129 & 3.5805 &2.2266 &1145.0297\\ 
78  & 1.3896 & 1.5560 & 1.6807 & 1.8527 & 2.1487 & 
2.5420 & 3.4796 &2.2316 &  1171.3408\\ 
79  & 1.3930 & 1.5669 & 1.6900 & 1.8486 & 2.1391 & 
2.5279 & 3.4635 &2.2365 &1197.9091\\ 
80  & 1.3873 & 1.5528 & 1.6711 & 1.8431 & 2.1262 & 
2.5284 & 3.4714 &2.2414 &1224.7337\\ 
81  & 1.3923 & 1.5573 & 1.6734 & 1.8414 & 2.1218 & 
2.4825 & 3.4870 &2.2462 &1251.8140\\

 82  & 1.3857 & 1.5559 & 1.6767 & 1.8532 & 2.1422 & 
2.4974 & 3.4783 &2.2509 &1279.1492\\ 
83  & 1.3860 & 1.5579 & 1.6799 & 1.8568 & 2.1377 & 
2.5472 & 3.4902 &2.2556 &1306.7388\\ 
84  & 1.4178 & 1.5832 & 1.7032 & 1.8737 & 2.1635 & 
2.5825 & 3.5271 &2.2602 &1334.5820\\ 
85  & 1.4066 & 1.5793 & 1.7007 & 1.8776 & 2.1598 & 
2.5435 & 3.4956 &2.2647 & 1362.6781\\ 
86  & 1.4097 & 1.5776 & 1.6929 & 1.8562 & 2.1554 & 
2.5406 & 3.5896 &2.2692 &1391.0266\\ 
87  & 1.4173 & 1.5900 & 1.7129 & 1.8834 & 2.1823 & 
2.6082 & 3.6324 &2.2736 &1419.6267\\ 
88  & 1.4149 & 1.5912 & 1.7170 & 1.8921 & 2.1844 & 
2.6253 & 3.6058 &2.2780 &1448.4778\\ 
89 & 1.4159 & 1.5853 & 1.7026 & 1.8711 & 2.1765 & 
2.5753 & 3.5992 &2.2823 &1477.5794\\ 
90  & 1.4191 & 1.5880 & 1.7085 & 1.8794 & 2.1779 & 
2.5629 & 3.5503 &2.2865 &1506.9307\\ 
91  & 1.4419 & 1.6148 & 1.7353 & 1.9100 & 2.2180 & 
2.6173 & 3.6562 &2.2907 &1536.5313\\ 
92 & 1.4381 & 1.6061 & 1.7302 & 1.9133 & 2.2286 & 
2.6358 & 3.7350 &2.2949 &1566.3805\\ 
93  & 1.4378 & 1.6117 & 1.7277 & 1.8948 & 2.2087 & 
2.5941 & 3.6177 &2.2990 &1596.4777\\ 
94  & 1.4356 & 1.6133 & 1.7393 & 1.9103 & 2.2107 & 
2.5887 & 3.6782 &2.3030 &1626.8223\\ 
95  & 1.4363 & 1.6112 & 1.7321 & 1.9039 & 2.2049 & 
2.6048 & 3.6679 &2.3070 &1657.4138\\ 
96  & 1.4375 & 1.6138 & 1.7321 & 1.9096 & 2.2107 & 
2.6169 & 3.6031 &  2.3110 &1688.2517\\ 
97 & 1.4355 & 1.6057 & 1.7274 & 1.9046 & 2.2146 & 
2.6168 & 3.6890 &2.3149 &1719.3353\\ 
98  & 1.4445 & 1.6137 & 1.7390 & 1.9171 & 2.2214 & 
2.6184 & 3.6037 &2.3188 &1750.6642\\ 
99  & 1.4338 & 1.6128 & 1.7343 & 1.9024 & 2.2065 & 
2.6307 & 3.6025 &2.3226 &1782.2377\\ 
100  & 1.4568 & 1.6358 & 1.7637 & 1.9396 & 2.2606 & 
2.6495 & 3.6210 &2.3263 &1814.0555\\ 
101  & 1.4648 & 1.6426 & 1.7661 & 1.9444 & 2.2449 & 
2.6978 & 3.7069 &2.3301 &1846.1168\\ 
102 & 1.4579 & 1.6328 & 1.7594 & 1.9336 & 2.2381 & 
2.6473 & 3.7758 &2.3338 &1878.4214\\ 
103  & 1.4656 & 1.6398 & 1.7669 & 1.9521 & 2.2556 & 
2.6575 & 3.6712 &2.3374 &1910.9685\\ 
104  & 1.4652 & 1.6465 & 1.7742 & 1.9433 & 2.2514 & 
2.6530 & 3.7589 &2.3410 &1943.7578\\ 
105 & 1.4668 & 1.6380 & 1.7589 & 1.9397 & 2.2514 & 
2.6561 & 3.6646 &2.3446 &1976.7887\\ \hline
\end{array}$ \end{footnotesize}

\begin{footnotesize}
 $
 \begin{array}{|c|ccccccc||c|c|} 
 \hline
\textbf{\textit{n}}  & \textbf{0.15}  & \textbf{0.1} & \textbf{0.075}  & \textbf{0.05} & 
\textbf{0.025} & \textbf{0.01}  & \textbf{0.001} & \bf{a_n} & \bf{C_n } \\ \hline
 106  & 1.4653 & 1.6402 & 1.7629 & 1.9431 & 2.2407 & 
2.6590 & 3.7376 &2.3481 &2010.0608\\ 
107 & 1.4624 & 1.6360 & 1.7586 & 1.9389 & 2.2620 & 
2.6619 & 3.7710 &2.3516 &2043.5736\\ 
108  & 1.4611 & 1.6334 & 1.7577 & 1.9301 & 2.2403 & 
2.6498 & 3.8194 & 2.3551 &2077.3266\\ 
109  & 1.4634 & 1.6421 & 1.7680 & 1.9488 & 2.2502 & 
2.6373 & 3.7118 &2.3585 &2111.3193\\ 
110  & 1.4654 & 1.6381 & 1.7622 & 1.9379 & 2.2344 & 
2.6400 & 3.6154 & 2.3619 &2145.5513\\ 
111 & 1.4671 & 1.6525 & 1.7806 & 1.9570 & 2.2594 & 
2.6291 & 3.7433 &2.3652 &2180.0221\\ 
112  & 1.4914 & 1.6740 & 1.7934 & 1.9840 & 2.2849 & 
2.7172 & 3.7062 &2.3686 &2214.7313\\

113 & 1.4973 & 1.6782 & 1.8027 & 1.9823 & 2.3010 & 
2.7150 & 3.7486 &2.3719 &2249.6784\\ 
114  & 1.4936 & 1.6668 & 1.7955 & 1.9757 & 2.2761 & 
2.7000 & 3.7684 &2.3751 &2284.8629\\ 
115 & 1.4859 & 1.6651 & 1.7934 & 1.9717 & 2.2655 & 
2.6743 & 3.7925 &2.3783 &2320.2846\\ 
116  & 1.4931 & 1.6768 & 1.8051 & 1.9844 & 2.2938 & 
2.7092 & 3.7641 &2.3815 &2355.9428\\ 
117  & 1.4982 & 1.6737 & 1.8047 & 1.9934 & 2.2950 & 
2.7069 & 3.8541 &2.3847 &2391.8373\\ 
118  & 1.4916 & 1.6766 & 1.8010 & 1.9803 & 2.2962 & 
2.7081 & 3.7584 &2.3878 &2427.9675\\ 
119  & 1.4936 & 1.6765 & 1.8043 & 1.9907 & 2.2986 & 
2.7223 & 3.8190 &2.3909 &2464.3331\\ 
120  & 1.4914 & 1.6709 & 1.7969 & 1.9796 & 2.2913 & 
2.6762 & 3.6763 &2.3940 &2500.9337\\ 
121  & 1.4992 & 1.6804 & 1.8054 & 1.9775 & 2.2979 & 
2.7028 & 3.8052 &2.3970  &2537.7688\\ 
122  & 1.5008 & 1.6747 & 1.7973 & 1.9802 & 2.2697 & 
2.6783 & 3.7898 &2.4000 &2574.8381\\ 
123  & 1.4952 & 1.6757 & 1.8048 & 1.9860 & 2.2866 & 
2.7052 & 3.7794 &2.4030 &2612.1411\\ 
124  & 1.4959 & 1.6759 & 1.8052 & 1.9868 & 2.2933 & 
2.7154 & 3.6734 &2.4060 &2649.6775\\ 
125  & 1.5318 & 1.7153 & 1.8474 & 2.0254 & 2.3491 & 
2.7653 & 3.9034 &2.4089 &2687.4469\\ 
126  & 1.5294 & 1.7087 & 1.8409 & 2.0284 & 2.3443 & 
2.7871 & 3.8894 &2.4118 &2725.4489\\ 
127  & 1.5381 & 1.7214 & 1.8481 & 2.0238 & 2.3321 & 
2.7497 & 3.8923 &2.4147 &2763.6831\\ 
128  & 1.5309 & 1.7123 & 1.8400 & 2.0182 & 2.3288 & 
2.7619 & 3.8255 & 2.4176 &2802.1492\\ 
129  & 1.5409 & 1.7260 & 1.8572 & 2.0429 & 2.3626 & 
2.7835 & 3.7827 &2.4204 &2840.8467\\ 

 130  & 1.5261 & 1.7104 & 1.8356 & 2.0192 & 2.3348 & 
2.7624 & 3.9036 &2.4232 &2879.7754\\ 
131  & 1.5235 & 1.7010 & 1.8275 & 2.0085 & 2.3269 & 
2.7192 & 3.8243 &2.4260  &2918.9348\\ 
132  & 1.5336 & 1.7112 & 1.8436 & 2.0235 & 2.3412 & 
2.7569 & 3.7583 &2.4287 &2958.3246\\ 
133  & 1.5288 & 1.7149 & 1.8497 & 2.0347 & 2.3422 & 
2.7812 & 3.7633 &2.4315  &2997.9444\\ 
134  & 1.5335 & 1.7190 & 1.8476 & 2.0265 & 2.3437 & 
2.8052 & 3.8702 &2.4342 &3037.7939\\ 
135  & 1.5384& 1.7142 & 1.8415 & 2.0223 & 2.3308 & 
2.7512 & 3.8024 & 2.4369 &3077.8728\\ 
136  & 1.5266 & 1.7098 & 1.8395 & 2.0177 & 2.3359 & 
2.7451 & 3.8569 &2.4395 &3118.1807\\ 
137 & 1.5298 & 1.7110 & 1.8388 & 2.0191 & 2.3386 & 
2.7836 & 3.8683 &2.4422 &3158.7172\\ 
138  & 1.5304 & 1.7133 & 1.8412 & 2.0182 & 2.3396 & 
2.7247 & 3.7756 &2.4448 &3199.4820\\ 
139  & 1.5280 & 1.7143 & 1.8458 & 2.0307 & 2.3437 & 
2.8075 & 3.8780 &2.4474 &3240.4748\\ 
140  & 1.5367 & 1.7242 & 1.8548 & 2.0318 & 2.3468 & 
2.7696 & 3.8377 &2.4500  &3281.6953\\ 
141  & 1.5285 & 1.7122 & 1.8489 & 2.0288 & 2.3198 & 
2.7296 & 3.9228 &2.4526 &3323.1431\\ 
142  & 1.5277 & 1.7122 & 1.8403 & 2.0260 & 2.3461 & 
2.7479 & 3.8630 &2.4551 &3364.8179\\ 
143  & 1.5703 & 1.7560 & 1.8874 & 2.0776 & 2.4052 & 
2.8404 & 3.9334 &2.4576 &3406.7194\\ 
144  & 1.5686 & 1.7593 & 1.8872 & 2.0716 & 2.3878 & 
2.8007 & 4.1252 & 2.4601 &3448.8472\\ 
145 & 1.5750 & 1.7631 & 1.8959 & 2.0823 & 2.3833 & 
2.8230 & 4.0015 &2.4626 &3491.2011\\ 
146  & 1.5679 & 1.7588 & 1.8976 & 2.0884 & 2.4081 & 
2.8252 & 3.8971 &2.4651 &3533.7808\\ 
147  & 1.5641 & 1.7590 & 1.8962 & 2.0794 & 2.4078 & 
2.8128 & 3.9899 &2.4675 &3576.5858\\ 
148  & 1.5628 & 1.7477 & 1.8827 & 2.0588 & 2.3786 & 
2.8044 & 4.0139 & 2.4699 &3619.6160\\ 
149  & 1.5781 & 1.7605 & 1.8931 & 2.0821 & 2.4324 & 
2.8661 & 4.0380 &2.4723 &3662.8710\\ 
150  & 1.5693 & 1.7627 & 1.8942 & 2.0844 & 2.4212 & 
2.8560 & 3.9064 &2.4747 &3706.3505\\ 
151  & 1.5664 & 1.7507 & 1.8861 & 2.0802 & 2.4117 & 
2.8842 & 4.0000 &2.4771 &3750.0543\\ 
152  & 1.5685 & 1.7542 & 1.8842 & 2.0725 & 2.4020 & 
2.8504 & 3.9861 &2.4795 &3793.9820\\ 
153  & 1.5670 & 1.7486 & 1.8777 & 2.0670 & 2.3894 & 
2.8398 & 3.9350 &2.4818 &3838.1333\\ 
 154  & 1.5736 & 1.7672 & 1.8997 & 2.0929 & 2.4096 & 
2.8565 & 3.9442 &2.4841 &3882.5079\\ 
155  & 1.5720 & 1.7637 & 1.8945 & 2.0866 & 2.4006 & 
2.8125 & 3.8631 &2.4864 & 3927.1056\\ 
156  & 1.5713 & 1.7670 & 1.8988 & 2.0980 & 2.4320 & 
2.8729 & 3.9977 &2.4887 &3971.9261\\ 
157  & 1.5689 & 1.7527 & 1.8830 & 2.0711 & 2.3973 & 
2.8370 & 3.8880 &2.4910 &4016.9691\\ \hline
\end{array}%
$\end{footnotesize}

\begin{footnotesize}
$%
 \begin{array}{|c|ccccccc||c|c|} 
 \hline
\textbf{\textit{n}}  & \textbf{0.15}  & \textbf{0.1} & \textbf{0.075}  & \textbf{0.05} & 
\textbf{0.025} & \textbf{0.01}  & \textbf{0.001} & \bf{a_n} & \bf{C_n } \\ \hline
158  & 1.5723 & 1.7604 & 1.8990 & 2.0775 & 2.4074 & 
2.8411 & 3.9726 &2.4932 &4062.2343\\ 
159  & 1.5728 & 1.7598 & 1.8960 & 2.0764 & 2.3939 & 
2.8315 & 4.0027 &2.4955 &4107.7214\\ 
160  & 1.5662 & 1.7526 & 1.8812 & 2.0655 & 2.4023 & 
2.8322 & 4.0344 &2.4977 &4153.4301\\ 
161 & 1.5673 & 1.7568 & 1.8910 & 2.0793 & 2.4079 & 
2.8703 & 3.9389 &2.4999 &4199.3603\\ 
162  & 1.5670 & 1.7538 & 1.8878 & 2.0758 & 2.4150 & 
2.8757 & 4.0885 &2.5021 &4245.5115\\ 
163  & 1.5731 & 1.7620 & 1.8952 & 2.0779 & 2.3952 & 
2.8549 & 3.9610 &2.5043 &4291.8836\\ 
164  & 1.5638 & 1.7507 & 1.8851 & 2.0779 & 2.4150 & 
2.8065 & 3.8816 &2.5064 &4338.4763\\ 
165  & 1.5700 & 1.7558 & 1.8872 & 2.0719 & 2.3851 & 
2.8230 & 3.9072 &2.5086 &4385.2894\\ 
166  & 1.5566 & 1.7460 & 1.8801 & 2.0689 & 2.3815 & 
2.8411 & 3.9225 &2.5107 &4432.3224\\ 
167 & 1.6161 & 1.8091 & 1.9469 & 2.1541 & 2.4844 & 
2.9439 & 4.0744 &2.5128 &4479.5753\\ 
168  & 1.6217 & 1.8117 & 1.9473 & 2.1437 & 2.4647 & 
2.9140 & 4.0022 &2.5150 &4527.0477\\ 
169  & 1.6170 & 1.8016 & 1.9385 & 2.1341 & 2.4764 & 
2.9199 & 4.0265 &2.5170 &4574.7394\\ 
170  & 1.6114 & 1.8041 & 1.9389 & 2.1331 & 2.4608 & 
2.9014 & 4.1730 & 2.5191 &4622.6501\\ 
171  & 1.6161 & 1.8114 & 1.9454 & 2.1368 & 2.4676 & 
2.9217 & 4.0926 &2.5212 &4670.7796\\ 
172  & 1.6106 & 1.7986 & 1.9399 & 2.1326 & 2.4624 & 
2.8999 & 4.1438 &2.5232 &4719.1276\\ 
173  & 1.6192 & 1.8214 & 1.9553 & 2.1474 & 2.4929 & 
2.9706 & 4.0761 &2.5253 &4767.6939\\ 
174  & 1.6157 & 1.8069 & 1.9431 & 2.1289 & 2.4482 & 
2.8809 & 4.0472 &2.5273  &4816.4782\\ 
175  & 1.6082 & 1.8052 & 1.9447 & 2.1369 & 2.4672 & 
2.9281 & 3.9707 &2.5293 &4865.4804\\ 
176  & 1.6173 & 1.8097 & 1.9474 & 2.1422 & 2.4750 & 
2.9213 & 4.0997 &  2.5313 &4914.7001\\ 
177  & 1.6078 & 1.7979 & 1.9427 & 2.1378 & 2.4589 & 
2.9142 & 4.0648 &2.5333 &4964.1371\\ 

 178  & 1.6102 & 1.8069 & 1.9417 & 2.1439 & 2.4849 & 
2.9138 & 3.9984 &2.5353 &5013.7912\\ 
179  & 1.6150 & 1.8043 & 1.9347 & 2.1237 & 2.4533 & 
2.9027 & 4.0803 &2.5372 &5063.6621\\ 
180  & 1.6133 & 1.8025 & 1.9414 & 2.1374 & 2.4735 & 
2.9096 & 3.9865 &2.5392 &5113.7496\\ 
181  & 1.6080 & 1.7933 & 1.9204 & 2.1121 & 2.4501 & 
2.8787 & 4.0807 &2.5411 &5164.0535\\ 
182  & 1.6149 & 1.8081 & 1.9445 & 2.1443 & 2.4809 & 
2.9197 & 4.0438 &2.5430 &5214.5735\\ 
183  & 1.6136 & 1.8006 & 1.9423 & 2.1373 & 2.4674 & 
2.8889 & 4.2658 &2.5450 &5265.3095\\ 
184  & 1.6041 & 1.7966 & 1.9325 & 2.1395 & 2.4730 & 
2.9267 & 4.1624 &2.5469 &5316.2611\\ 
185 & 1.6135 & 1.8104 & 1.9452 & 2.1380 & 2.4619 & 
2.8934 & 4.3550 &2.5488 &5367.4282\\ 
186 & 1.6122 & 1.8053 & 1.9433 & 2.1357 & 2.4744 & 
2.9524 & 4.0716 &2.5506 &3 5418.8106\\ 
187  & 1.6162 & 1.8079 & 1.9412 & 2.1267 & 2.4492 & 
2.9129 & 4.0421 &2.5525  &5470.4079\\ 
188  & 1.6070 & 1.7985 & 1.9401 & 2.1441 & 2.4866 & 
2.9596 & 4.0490 &2.5544 &5522.2200\\ 
189  & 1.6244 & 1.8170 & 1.9510 & 2.1456 & 2.4839 & 
2.9188 & 4.0255 &2.5562 &5574.2468\\ 
190  & 1.6234 & 1.8161 & 1.9584 & 2.1561 & 2.4949 & 
2.9272 & 4.0296 &2.5580 &5626.4878\\ 
191  & 1.6099 & 1.8018 & 1.9372 & 2.1328 & 2.4302 & 
2.8630 & 4.0091 &2.5599 &5678.9431\\ 
192  & 1.6190 & 1.8098 & 1.9478 & 2.1408 & 2.4843 & 
2.9319 & 4.1542 & 2.5617 &5731.6122\\
193  & 1.6219 & 1.8135 & 1.9520 & 2.1423 & 2.4679 & 
2.9339 & 4.1218 &2.5635 &5784.4951\\ 
194  & 1.6181 & 1.8069 & 1.9465 & 2.1347 & 2.4615 & 
2.8779 & 4.0501 &2.5653 &5837.5914\\ 
195  & 1.6056 & 1.8000 & 1.9323 & 2.1229 & 2.4512 & 
2.8756 & 4.1073 &2.5671 &5890.9011\\ 
196  & 1.6090 & 1.8070 & 1.9426 & 2.1445 & 2.4596 & 
2.9005 & 4.1082 &2.5688 &5944.4238\\ 
197  & 1.6227 & 1.8158 & 1.9551 & 2.1444 & 2.4778 & 
2.9440 & 4.1214 &2.5706 &5998.1594\\ 
198  & 1.6224 & 1.8160 & 1.9520 & 2.1429 & 2.4777 & 
2.9228 & 4.1046 &2.5724 &6052.1077\\ 
199  & 1.6171 & 1.8062 & 1.9428 & 2.1411 & 2.4806 & 
2.9012 & 4.0380 &2.5741 &6106.2685\\ 
200  & 1.6534 & 1.8537 & 1.9956 & 2.1848 & 2.5180 & 
2.9667 & 4.2543 &2.5758 &6160.6415\\ 
\hline 
\end{array}%
$
\end{footnotesize}

\begin{footnotesize}
$%
\begin{array}{|c|ccccccc||c|c|} 
\hline
\textbf{\textit{n}}  & \textbf{0.15}  & \textbf{0.1} & \textbf{0.075}  & \textbf{0.05} & 
\textbf{0.025} & \textbf{0.01}  & \textbf{0.001} & \bf{a_n} & \bf{C_n } \\ \hline
250 & 1.6854 & 1.8946 & 2.0403 & 2.2348 & 2.5882 & 3.0200 & 4.4252 &2.6521 & 9145.8 \\
500 & 1.8803 & 2.1020 & 2.2570 & 2.4775 & 2.8403 & 3.3546 & 4.6692 &2.8782 &31419.2  \\
750 & 1.9601 & 2.1819 & 2.3448 & 2.5749 & 2.9473 & 3.4903 & 4.8916 &3.0038 &65019.6 \\
1000 & 2.0218 & 2.2488 & 2.4066 & 2.6323 & 3.0151 & 3.5407 & 4.9433 &3.0902 &109204 \\
10000 & 2.7769 & 3.0512 & 3.2332 & 3.4776 & 3.8861 & 4.4679 & 5.6259 &3.7190 &7439183 \\
\hline
\end{array}	$
\end{footnotesize}

\section{Acknowledgements}
The author wishes to thank Enrique Cabaña for your help in the preparation of this work.

\section{References}

\begin{enumerate}
\item Anderson, T. W. and Darling, D. A. (1954). A test of goodness of fit. 
\textit{J. Amer. Statist. Assoc.} 49, 765--769. 

\item del Barrio, E., Cuesta Albertos, J. A., Matr\'an, C.
and Rodr\'iguez Rodr\'iguez, J. (1999). Tests of fit based on the
L2-Wasserstein distance. \textit{Ann. Statist}. 27, 1230--1239. 

\item Cram\'er, H. (1928). On the composition of elementary
errors. Second paper: Statistical applications. \textit{Skand. Aktuartidskr.}%
 11, 141--180. 

\item Csörgo%
${{}^3}$%
, S. and Faraway, J. (1996). The Exact and Asymptotic Distributions
of Cram\'er--von Mises Statistics \textit{J. R. Statist. Soc}. B 58(1), 
221--234. 

\item Durbin, J. (1973). Weak convegence of the sample distribution when
parameters are estimated. \textit{Ann. Statist.} 1, 219. 

\item Fisher, R. A., (1930). The moments of the distribution for normal
samples of measures of departure from normality. \textit{Proceedings of
Royal Society}, A, 130, 16.

\item Gan, F. F., Koelher, K. J. (1990). Goodness of fit tests based on P--P
probability plots. \textit{Technometrics} 32, 289--303.

\item Kalemkerian, J., (2017). An integral formula for the distribution of
self-normalized Gaussian random samples. \textit{Communications in
Statistics. Theory and Methods} 46(10), 4671--4685.

\item Lockhart, R. A. and Stephens, M. A., (1998). \textit{The probability
plot: Test of fit based on the correlation coefficient.} In \textit{Order
statistics: Applications. Handbook of Statistics} 17, 453--473. North
Holland, Amsterdam.

\item Pearson, K. (1895). Contributions to the mathematical theory of
evolution. \textit{Philosophical Transactions of the Royal Society}
91, 343.

\item Pearson, E. S., (1930). A further development of tests for normality. 
\textit{Biometrika} 22, 239--249. 

\item Pettitt, A. N. (1976). Cram\'er--von Mises statistics for
testing normality with censored samples. \textit{Biometrika}, 63(3),
475--481.

\item Pettitt, A. N. and Stephens, M. A. (1976). Modified Cram\'er--von Mises statistics with censored samples. \textit{Biometrika}
63(2), 291--298. 

\item Shapiro, S. S. and Wilk, M. B. (1965). An analysis of variance test
for normality (complete samples). \textit{Biometrika} 64, 415--418. 

\item Shorack, G. and Wellner, J. (1982). Limit theorems and inequalities for the uniform empirical processes indexed by intervals.
\textit{The Annals of Probability} 10(3), 639--652.

\item Skorokhod, A. (1956). Limit theorem for stochastic processes. \textit{Theor. Probability Appl.} 1, 261--290.

\item Smirnov, N. V. (1936). Sur la distribution de $w^{2}$ (Crit\'erium de M. R. von Mises). \textit{Comptes Rendus de l'Acad\'emie 
des Sciences} 202, 449--452

\item Smirnov, N. V. (1937). Sur la distribution de $w^{2}$ (Crit\'erium de M. R. von Mises). \textit{Matematicheskij Sbornik}
(in Russian with French summary) 2,
973--993.

\item Stephens, M. A. (1974). EDF Statistics for Goodness of Fit and Some
Comparisions. \textit{JASA} 79, 730--737.

\item Stephens, M. A. (1986). Tests based on EDF statistics. In R. B. D'Agostino and M. A. Stephens, eds.. \textit{Goodness of Fit Techniques}, North-Holland, Amsterdam. 

\item de Wet and Venter, J. (1973). Asymptotic distributions for quadratic
forms with applications to test of fit. \textit{Ann. Statist.} 2, 380--387.

\item von Mises, R. (1931). \textit{Wahrscheinlichkeitsrechnung}, Vienna.

\item Williams, P. (1935). Note on the sampling distribution of $\sqrt{%
\beta_{2}}$, where the population is normal. \textit{Biometrika} 27,
269--271.
\end{enumerate}

\end{document}